\documentclass[12pt,reqno]{amsart}
\usepackage{amssymb}
\usepackage[all]{xy}
\usepackage{hyperref}
\usepackage{enumitem}

\theoremstyle{plain}
\newtheorem{prop}{Proposition}
\newtheorem{thm}[prop]{Theorem}
\newtheorem{cor}[prop]{Corollary}
\newtheorem{lem}[prop]{Lemma}

\newtheorem*{step}{Claim}

\theoremstyle{definition}
\newtheorem{defi}[prop]{Definition}

\theoremstyle{remark}

\newtheorem{rem}[prop]{Remark}
\newtheorem{example}[prop]{Example}
\newtheorem{constr}[prop]{Construction}

\newtheorem{claim}{Claim}[prop]

\numberwithin{prop}{section}
\numberwithin{ques}{section}
\numberwithin{equation}{section}

\DeclareMathOperator{\HNN}{HNN}
\DeclareMathOperator{\ourHNN}{\widetilde{HNN}}
\DeclareMathOperator{\pHNN}{\overline{HNN}}

\DeclareMathOperator{\Subgr}{Subgr}

\DeclareMathOperator{\Sym}{Sym}

\DeclareMathOperator{\Ker}{Ker}

\DeclareMathOperator{\iid}{id}

\DeclareMathOperator{\Stab}{Stab}

\newcommand{\Bhat}{\hat{B}}
\newcommand{\That}{\hat{T}}
\newcommand{\alphahat}{\hat{\alpha}}
\newcommand{\gammahat}{\hat{\gamma}}
\newcommand{\varphihat}{\hat{\varphi}}
\newcommand{\psihat}{\hat{\psi}}
\newcommand{\xhat}{\hat{x}}
\newcommand{\yhat}{\hat{y}}
\newcommand{\Yhat}{\hat{Y}}
\newcommand{\Ahat}{\hat{A}}
\newcommand{\ahat}{\hat{a}}
\newcommand{\bhat}{\hat{b}}
\newcommand{\Xhat}{\hat{X}}

\newcommand{\ca}[1]{\mathcal{#1}}

\newcommand{\F}{\mathbb{F}}
\newcommand{\N}{\mathbb{N}}

\newcommand{\bfB}{\mathbf{B}}
\newcommand{\bfA}{\mathbf{A}}
\newcommand{\bfG}{\mathbf{G}}

\def\bfAhat{{\hat\bfA}}

\def\bfBhat{{\hat\bfB}}

\newcommand{\suchthat}{\mathop{|\,}}

\newcommand{\trsl}{\bullet}

\newcommand{\Abar}{\bar{A}}
\newcommand{\Bbar}{\bar{B}}
\newcommand{\Fbar}{\bar{F}}
\newcommand{\Gbar}{\bar{G}}
\newcommand{\Mbar}{\bar{M}}
\newcommand{\Tbar}{\bar{T}}

\newcommand{\psibar}{\bar{\psi}}

\newcommand{\U}{\ca{U}}

\newcommand{\caL}{\ca{L}}

\newcommand{\tG}{\tilde{G}}
\newcommand{\tF}{\tilde{F}}

\newcommand{\tN}{\tilde{N}}
\newcommand{\tS}{\tilde{S}}
\newcommand{\tU}{\tilde{U}}

\newcommand{\trho}{\tilde{\rho}}

\newcommand{\argu}{\hbox to 7truept{\hrulefill}}

\newcommand{\G}{\ca{G}}

\newcommand{\Hc}{\ca{H}}
\newcommand{\A}{\ca{A}}
\newcommand{\B}{\ca{B}}
\newcommand{\C}{\ca{C}}

\newcommand{\normal}{\triangleleft}

\newcommand\myatop[2]{\genfrac{}{}{0pt}{}{#1}{#2}}

\makeatletter
\def\moverlay{\mathpalette\mov@rlay}
\def\mov@rlay#1#2{\leavevmode\vtop{%
   \baselineskip\z@skip \lineskiplimit-\maxdimen
   \ialign{\hfil$\m@th#1##$\hfil\cr#2\crcr}}}
\newcommand{\charfusion}[3][\mathord]{
    #1{\ifx#1\mathop\vphantom{#2}\fi
        \mathpalette\mov@rlay{#2\cr#3}
      }
    \ifx#1\mathop\expandafter\displaylimits\fi}
\makeatother

\newcommand{\dotcup}{\charfusion[\mathbin]{\cup}{\cdot}}
\newcommand{\bigdotcup}{\charfusion[\mathop]{\bigcup}{\cdot}}

\title{Relatively projective pro-$p$ groups}
\author{ Dan Haran and Pavel A. Zalesskii}
\date{\today}

\address{Dan Haran\\
Sackler School of Mathematical Sciences\\
Tel Aviv University\\
Tel Aviv\\
Israel}
\email{haran@tauex.tau.ac.il}

\address{P. A. Zalesski\u i\\
Department of Mathematics\\
University of Brasilia\\
70.910 Brasilia DF\\
Brazil}
\email{pz@mat.unb.br}
\begin{document}

\maketitle



\section{Introduction}

It is well-known that
the Schreier theorem does not hold for free profinite groups,
i.e.,
a subgroup of a free profinite group need not be free.
Subgroups of free profinite groups are groups of cohomological dimension 1;
they are called projective,
as they satisfy the same universal property as projective modules.

Similarly, the Kurosh Subgroup Theorem
does not hold for free products of profinite groups.
This led the first author to introduce the notion of
a profinite group projective relative to a family of its subgroups,
closed under conjugation.
Then \cite[Corollary 5.4]{Haran} has shown
that a subgroup $G$ of a free profinite product
$H = \coprod_{x\in X} H_x$
is projective relative to the family of subgroups
$\G = \{H_x^h\cap G \suchthat h \in H, x\in X\}$. 
An analogous result holds for subgroups of free pro-$p$ products.

However,
the converse, that is,
whether a profinite group $G$,
projective relative to a continuous family
$\G$ of its subgroups,
is a subgroup of a free product 
in the above manner,
has been treated only partially.

Assuming that $\G$ is a continuous family closed under conjugation,
we have
$\G=\{G_t \suchthat t\in T\}$,
where $T$ is a profinite space
on which $G$ continuously acts
so that the map $t \mapsto G_t$
is $G$-equivariant
and injective on $\{t \in T \suchthat G_t \ne 1\}$
(\cite[Lemma 3.5]{Haran}).
Then the answer to the above question is positive,
if $T$ has a closed subset $T_0$ of representatives of the $G$ orbits
(\cite[Theorem 9.5]{Haran}).
In particular, this is the case if $G$ is second countable
(\cite[Theorem 8.5]{Haran}).
Moreover, if $G$ is a second countable pro-$p$ group,
then $G$ is even a free pro-$p$ product
$G = \coprod_{t\in T_0} G_t\amalg F$,
where $F$ is a free pro-$p$ group
(\cite[Corollary 9.6]{Haran}).

On the other hand, 
the second author has produced
an example of a pro-$p$ group $G$,
projective relative to a family
$\G=\{G_t \suchthat t\in T\}$,
where the action of $G$ on $T$ has
no closed subset of representatives of the $G$-orbits.
But $G$ is, nevertheless,
a subgroup of a free pro-$p$ product
in the above manner
(\cite[Proposition 4.6]{jalg}).

The objective of this paper is to give a positive answer
to the above question for pro-$p$ groups.

First we note
(Lemma \ref{sub free prod})
that if $G$ is a subgroup of a free profinite (or pro-$p$) product
$H = \coprod_{x\in X} H_x$,
then,
denoting
$\G = \{H_x^h\cap G \suchthat h\in H, x\in X\}$,
\begin{itemize}
\item[(a)]
there is a profinite space $T$
on which $G$ acts so that
the family of the stabilizers $\{G_t \suchthat t \in T\}$
of this action satisfies
\begin{itemize}
\item[(a1)]
the map $t \mapsto G_t$
is injective on $T' = \{t \in T \suchthat G_t \ne 1\}$;
\item[(a2)]
$\G \smallsetminus \{1\} = \{G_t \suchthat t \in T'\}$;
\end{itemize}
\end{itemize}
and, as remarked,
\begin{itemize}[resume]
\item[(b)]
$G$ is projective relative to $\G$.
\end{itemize}
Conversely:

\begin{thm}\label{smain}
Let $G$ be a pro-$p$ group $G$
and $\G \subseteq \Subgr(G)$.
Assume (a), (b).
Then there exists an embedding
$\zeta \colon G\to \Gbar = L \amalg F$
into a free pro-$p$ product of a copy $L$ of $G$
and a free pro-$p$ group $F$
such that
$\{\zeta(G_t) \suchthat t \in T'\}
=
\{L^\sigma\cap \zeta(G) \suchthat \sigma\in \Gbar\} \smallsetminus \{1\}$.
\end{thm}

Theorem \ref{smain} shows that
relatively projective pro-$p$ groups can be characterized as
subgroups of free pro-$p$ products.
However, the factors of the free product in Theorem \ref{smain}
are as complicated as the original group.
Our methods however show that one can find a simpler free product,
provided the map $t \mapsto G_t$ does not vary too much,
in a sense:

\begin{thm}\label{main small}
Under the assumptions of Theorem \ref{smain}
let $\rho \colon G\to L$ be a homomorphism into a pro-$p$ group
such that $\rho_{|G_t}$ is a monomorphism for each $t\in T$.
Then there exists an embedding
$\zeta \colon G\to \Gbar = L \amalg F$
into a free pro-$p$ product of $L$
and a free pro-$p$ group $F$
such that
$\{\zeta(G_t) \suchthat t \in T'\}
=
\{L^\sigma\cap \zeta(G) \suchthat \sigma\in \Gbar\} \smallsetminus \{1\}$.
\end{thm}

Using Theorem \ref{main small} one can try to index the family of subgroups
$\G=\{G_t \suchthat t\in T\}$
of a relatively projective pro-$p$ group $G$
such that
$G$ embeds into the free pro-$p$ product of the new family
according to the following

\begin{thm}\label{into free product} 
Let $G$ be a pro-$p$ group projective relative to a continuous family
$\G=\{G_t \suchthat t\in T\}$
of its subgroups closed under conjugation.
Let $\rho \colon G\to L=\coprod_{x\in X} L_x $ be a homomorphism
into a free pro-$p$ product
such that
$\rho_{|G_t}$ is a monomorphism into some conjugate of $L_x$.
Then there exists a monomorphism $\zeta \colon G\to L\amalg F$
such that
$\zeta(G_t)$ is conjugate to a subgroup of some $L_x$.
\end{thm}

Using the notion of pile (Definition \ref{pile})
we have another characterization of
relative projective groups:

\begin{thm}\label{inverse limit of free products}
Let $G$ act on a profinite space $T$ 
and let $\G = \{G_t \suchthat t\in T\}$.
Then the following statements are equivalent:

\begin{enumerate}
\item[(i)]
$G$ is projective relatively to $\G$ and (a1), (a2) above hold;

\item[(ii)]
$(G,T)$ is a projective pile;

\item[(iii)]
there exists a profinite space $\That$
that also satisfies (a1), (a2)
and $(G,\That) = \varprojlim_i (G_i,T_i)$,
where
$G_i = (\coprod_{x \in X_i} (G_i)_x) \amalg F_i$,
is a second countable free pro-$p$ product
with
$X_i$ a closed set of representatives of the $G_i$-orbits in $T_i$
and $F_i$ a free pro-$p$ group;
in particular, if $i \le j$, then
the map $\pi_{ij}$ of the inverse system maps every $(G_j)_x$
into a conjugate of $(G_i)_{\pi_{ij}(x)}$ in $G_i$;

\item[(iv)]
$G$ acts on a pro-$p$ tree with $\G \smallsetminus \{1\}$
being a family of nontrivial stabilizers of  vertices
and with trivial edge stabilizers. 
\end{enumerate}
\end{thm}

To illustrate Theorem \ref{main small} consider the following

\begin{example} 
Let $G$ be a pro-$p$ group and let $\G \subseteq \Subgr(G)$.
Assume (a), (b) above
and assume that the point stabilizers $G_t$ are abelian.
By Theorem \ref{inverse limit of free products},
$G$ is an inverse limit of free pro-$p$ products $G_i$ of abelian groups.
Each $G_i$ is an inverse limit of
free pro-$p$ products $G_{ij}$ of finitely many abelian factors.
As $[G_{ij},G_{ij}]$ intersects factors trivially
(say, because the projection on a factor is injective on that factor
and maps the commutator into $1$)
one deduces that $[G,G]\cap G_t=1$ for each $t\in T$.
Let $L = G/[G,G]$ be the abelianization of $G$. 
Then the quotient map
$\rho \colon G \to L$ 
is injective on $G_t$, for every $t \in T$.
Thus by Theorem \ref{main small} there exists an embedding 
$\zeta \colon G\to L\amalg F$ such that 
$\zeta(G_t)$ is conjugate to 
a subgroup of $L$ for every $t\in T$.

\end{example}

\bigskip

Let $\C$ be a family of finite groups 
closed under quotients, subgroups, and extensions.
In most of our application $\C$ will be the family of $p$-groups,
for a fixed prime $p$, but sometimes a more general treatment
seems to be more appropriate.

Given a subset $X$ of a profinite group $H$,
we denote by
$\langle X \rangle^H$
the smallest closed normal subgroup of $H$ containing $X$.

\medskip
\noindent{\bf Acknowledgements.} This paper was written when the second author visited Department of Pure Mathematics and Mathematical Statistics of University of Cambridge; the second author thanks Henry Wilton and the department for the hospitality.

\section{Profinite spaces and action}

\begin{rem}\label{partition definition}
Let $T$ be a profinite space.
A \textbf{partition} of $T$
is a finite family $X = \{T_i\}_{i=1}^n$ of nonempty clopen subsets of $T$
such that $T = \bigdotcup_{i=1}^n T_i$ (a disjoint union).
It induces a continuous map $\varphi \colon T \to X$,
by mapping every $t \in T_i$ onto $T_i \in X$.
Conversely, a continuous map
$\varphi \colon T \to X$
into a finite discrete space $X$
defines a partition, namely,
$\{\varphi^{-1}(\{x\}) \suchthat x \in \varphi(T) \subseteq X\}$.

A partition 
$Y = \{U_j\}_{j=1}^m$ is
\textbf{finer} than
(or a \textbf{refinement} of) $X$
if for every $j$ there is $i$ such that $U_j \subseteq T_i$.
\end{rem}

\begin{rem}\label{blocks}
Let $G$ be a profinite group continuously acting on 
a profinite space $T$.
If $Z$ is a closed subset of $T$, then
$$
\Stab_G(Z) = \{g \in G \suchthat Z^g = Z\}
$$
is a closed subgroup of $G$.
In particular,
$\Stab_G(t) = \{g \in G \suchthat t^g = t\}$ is a closed subgroup of $G$,
for every $t \in T$.
If $Z$ is clopen, then
$\Stab_G(Z)$ is open.

A partition $X = \{T_i\}_{i=1}^n$ of $T$
is a \textbf{$G$-partition}
if for every $i$ and every $g \in G$
either $T_i^g = T$ or $T_i \cap T_i^g = \emptyset$.
It defines an obvious $G$-action on $X$.
A partition $X$ of $T$ is a $G$-partition 
if and only if the induced map $\varphi \colon T \to X$ is $G$-equivariant.

An element of a $G$-partition is called a \textbf{$G$-block}.
\end{rem}

\begin{lem}\label{Stab}
Let $G$ be a profinite group continuously acting on 
a profinite space $T$.
Let $t \in T$ and let
$\B(t) = \{U \suchthat U \text{ is a $G$-block, } t \in U\}$.
\begin{itemize}
\item[(a)]
Every partition of $T$ can be refined by a $G$-partition.
\item[(b)]
$T$ is the inverse limit of its $G$-partitions.
\item[(c)]
$\B(t)$ is a basis of neighborhoods of $t$.
\item[(d)]
$\{t\} = \bigcap_{U \in \B(t)} U$.
\item[(e)]
If $U \subseteq T$ is a $G$-block
and $Z \subseteq T$ is a non-empty closed subset,
then
$\Stab_G(Z) \le \Stab_G(U)$;
in particular,
$\Stab_G(t) \le \Stab_G(U)$ for every $t \in U$;

\item[(f)]
$\Stab_G(t) = \bigcap_{U \in \B(t)} \Stab_G(U)$.
\end{itemize}
\end{lem}

\begin{proof}
(a) \cite[Lemma 7.1.1]{HJ}, Claim B.

(b) \cite[Lemma 7.1.1]{HJ}.

(c)
Follows from (a).

(d)
Follows from (c).

(e)
Let $g \in \Stab_G(Z)$.
Then $\emptyset \ne Z = Z^g \subseteq U \cap U^g$,
so $U \cap U^g \ne \emptyset$,
and hence $U = U^g$; thus $g \in \Stab_G(U)$.

(f)
By (e),
$\Stab_G(t) \le \bigcap_{U \in \B(t)} \Stab_G(U)$.
If $g \in \bigcap_{U \in \B(t)} \Stab_G(U)$,
then
$t^g \in \bigcap_{U \in \B(t)} U^g = \bigcap_{U \in \B(t)} U = \{t\}$,
by (d),
so $g \in \Stab_G(t)$.
\end{proof}

The next lemma is an analogue of \cite[Lemma 2.1.3]{HJ}.

\begin{lem}\label{G-partition}
Let a \textbf{finite} group $G$ continuously act on a profinite space $T$
and let $\mathcal{P}$ be a partition of $T$.
Then there is a $G$-partition
$T = \bigdotcup_{i=1}^n T_i$ of $T$, 
finer than $\mathcal{P}$,
such that for every $i$ there is $t_i \in T_i$
with
$\Stab_G(T_i) = G_{t_i}$. 
\end{lem}

\begin{proof}
%
By induction on the size of $\G = \{G_t \suchthat t \in T\}$.

For every maximal $\Gamma \in \G$ let
$C(\Gamma) = \{t \in T \suchthat G_t = \Gamma\}$.
This is a closed subset of $T$,
because it is the complement of the open subset
$\bigcup_{\Gamma' \in \G, \, \Gamma' \ne \Gamma}
\{t \in T \suchthat G_t \le \Gamma'\}$
of $T$.
If $g \in G$, then $\Gamma^g$ is also maximal in $\G$,
and 
$C(\Gamma)^g = C(\Gamma^g) = \{t \in T \suchthat G_t = \Gamma^g\}$.

The union
$C = \bigcup_{g \in G} C(\Gamma)^g$
is also closed; it is $G$-invariant.
Clearly,
$C = \bigdotcup_{g \in G/S} C(\Gamma)^g$,
where $S = N_G(\Gamma)$ is the stabilizer of $C(\Gamma)$ in $G$.
Thus this is a $G$-partition of $C$.
Deduce that if $C' \subseteq C(\Gamma)$,
then $\Stab_G(C') = \Stab_S(C')$.

By Lemma~\ref{Stab}(a)
we may assume that
$\mathcal{P}$ is a $G$-partition.

By Lemma~\ref{Stab}(d),(f),
every $t \in C(\Gamma)$
is contained in an $S$-block $C_t$ of $C(\Gamma)$
with $\Stab_S(C_t) = \Stab_S(t) = \Stab_G(t) = \Gamma$.
As $C(\Gamma)$ is compact,
finitely many of these $C_t$ cover $C(\Gamma)$.
Their intersections constitute a partition $\mathcal{P}'$ of $C(\Gamma)$.
By Lemma~\ref{Stab}(a),
$C(\Gamma)$ has an $S$-partition
$C(\Gamma) = \bigdotcup_{i=1}^m C_i$,
finer than $\mathcal{P}'$
and finer than the partition induced by $\mathcal{P}$ on $C(\Gamma)$.

Choose $t' \in C_i$ and $t \in C(\Gamma)$ such that $C_i \subseteq C_t$.
By Lemma~\ref{Stab}(e),
$\Gamma = \Stab_S(t') \le \Stab_S(C_i)$
and
$\Stab_S(C_i) \le \Stab_S(C_t) = \Gamma$.
Thus
$\Stab_S(C_i) = \Gamma = G_t$ for every $t \in C_i$.

Now
$C = \bigdotcup_{g \in G/S} \bigdotcup_{i} C_i^g$
is a $G$-partition of $C$
and 
$\Stab_G(C_i^g) = \Stab_{S^g}(C_i^g) = G_t$
for every $t \in C_i^g$.
We write this partition as
$\{C_i\}_{i=1}^n$.

There are disjoint open subsets $T_1,\ldots, T_n$ of $T$ 
such that $C_i \subseteq T_i$ for every $i$.
Let $1 \le i \le n$.
Every $t \in C_i$ has a clopen neighborhood in $T_i$
such that
$G_{t'} \le \Gamma$ 
for every $t'$ in that neighborhood.
As $C_i$ is compact,
we may replace $T_i$ by a union of finitely many such neighborhoods,
so that $T_i$ is clopen in $T$ and
$G_{t'} \le \Gamma$,
for every $t' \in T_i$.

If $1 \le i,k \le n$, and
$C_i^g = C_k$, for some $g \in G$, then,
without loss of generality,
$T_i^g = T_k$,
otherwise replace $T_i$ by its subset
$V_i = \bigcap_{(j,h)} T_j^{h^{-1}}$,
where $(j,h)$ runs through all pairs of 
$1 \le j \le n$ and $h \in G$
such that $C_i^h = C_j$.
Indeed, if $C_i^g = C_k$, then
$$
V_i^g =
\Big(
\bigcap_{\myatop{(j,h)}{C_i^h = C_j}}
    T_j^{h^{-1}} \Big)^g =
\bigcap_{\myatop{(j,h)}{C_k^{g^{-1}h} = C_j}}
    T_j^{(g^{-1} h)^{-1}} =
\bigcap_{\myatop{(j,\sigma)}{C_k^{\sigma} = C_k}}
    T_j^{\sigma^{-1}} =
V_k.
$$
Thus
$U = \bigdotcup_{i=1}^n T_i$
is
is a $G$-invariant clopen subset of $T$,
and $\{T_i\}_{i=1}^n$ is its $G$-partition,
such that
$\Stab_G(T_i) = \Stab_G(C_i)$ for every $i$.
Choose $t_i \in C_i \subseteq T_i$,
then 
$\Stab_G(T_i) = G_{t_i}$ for every $i$.

Also $T' = T \smallsetminus U$ is a clopen $G$-invariant subset.
As the set
$\{G_t \suchthat t \in T'\}$
does not contain $\Gamma$,
its size is strictly smaller than $|\G|$.
Therefore, by induction on $|\G|$,
there is a partition of $T'$, which, together with 
the partition $\{T_i\}_{i=1}^n$ of $U$,
gives the required $G$-partition of $T$.
\end{proof}

Let $F(T)$, resp. $F(T,*)$,
denote the free pro-$\C$ group on a
profinite space $T$,
resp. pointed profinite space $(T,*)$
(\cite[Section 3.3]{RZ}).

\begin{lem}\label{centralizer}
Let $G$ be a pro-$\C$ group acting on a profinite space $T$
and let $\tG=F(T)\rtimes G$ be the induced semidirect product.
Then $C_{F(T)}(G)=\langle T^G\rangle= F(T^G)$,
where $T^G$ is the subspace of points of $T$ fixed by $G$.
\end{lem}

\begin{proof} 
First assume that $T$ is finite.

As
$C_{F(T)}(G) = \bigcap_{\sigma \in G} C_{F(T)}(\langle \sigma \rangle)$
and
$\bigcap_{\sigma \in G} F(T^{\langle \sigma \rangle}) =
F(\bigcap_{\sigma \in G} T^{\langle \sigma \rangle})$,
by \cite[Lemma 3.1.5]{HJ},
it suffices to prove the assertion for $G$ procyclic.
Since the centralizer is contained in the normalizer,
it suffices to prove in this case a stronger assertion,
namely, that
$N_{F(T)}(G) = F(T^G)$;
equivalently,
$N_{\tG}(G) = F(T^G) \times G$.

Write $T$ as the disjoint union
$T = \bigdotcup_{i=1}^n T_i$ of its $G$-orbits.
Then
$F(T) = \coprod_{i=1}^n F(T_i)$,
hence
$\tG = \coprod_{G,i} \tG_i$,
where $\tG_i = F(T_i) \rtimes G$ for each $i$.
By \cite[Proposition 2.8]{BPZ}
$N_{\tG}(G) = \coprod_{G,i} N_{\tG_i}(G)$.
Therefore by the pro-$p$ version of the Kurosh Subgroup Theorem \cite[Theorem 9.1.10]{RZ}
$N_{F(T)}(G) = \coprod_{i} N_{F(T_i)}(G)$.
If $G$ acts trivially on $T_i$, then
$N_{F(T_i)}(G) = F(T_i)$.
Thus it suffices to show that if $G$ acts nontrivially on $T_i$, 
then $N_{F(T_i)}(G) = 1$;
equivalently, $N_{\tG_i}(G) = G$.

Replacing $G$ by its image in $\Sym(T_i)$
we may assume that $G$ acts faithfully on $T_i$.
In particular,
$G$ is a finite cyclic group and has trivial stabilizers.
Thus, for some $t \in T_i$, the map $\sigma \mapsto t^\sigma$
is a bijection $G \to T_i$.
By \cite[Lemma 4.7.4]{HJ}, 
$\tG_i =\langle t \rangle \amalg G$.
Hence, by a theorem of Herfort and Ribes
(\cite[Theorem A]{HR 85}),
$N_{\tG_i}(G) = G$.

In the general case $T$ is an inverse limit of finite $G$-spaces,
$T = \varprojlim T_k$.
Then
$T^G = \varprojlim T^G_k$,
and
$F(T) = \varprojlim F(T_k)$
and $\tG = \varprojlim (F(T_k)\rtimes G)$.
Hence
$C_{F(T)}(G) = \varprojlim C_{F(T_k)}(G) = \varprojlim F(T^G_k) = F(T^G)$.
\end{proof}

\begin{lem}\label{free group}
Let $F = F(T,*)$ be a free pro-$\C$ group
on a pointed profinite space $(T,*)$.
Then there is a profinite space $Y$ 
such that
$F$ is $F(Y)$, the free pro-$\C$ group on $Y$.
\end{lem}

\begin{proof}
If $T$ is finite,
then $Y = T \smallsetminus \{*\}$ satisfies
$F = F(Y)$.
So assume that $T$ is infinite.
By \cite[3.5.12]{RZ},
$F(T,*)$ and $F(T)$ are free pro-$p$ groups of the same rank,
and hence $F(T,*)\cong F(T)$.
\end{proof}

\section{Relative projective groups}

Let $G$ be a profinite group and
$(\G, T)=\{G_t \suchthat t\in T\}$
be a family of subgroups indexed by a profinite space $T$.
Following \cite[Section 5.2]{R 2017} we say that
$(\G,T)$ is \textbf{continuous}
if for any open subgroup $U$ of $G$
the subset $\{ t\in T \suchthat G_x \le U\}$ is open.

For instance, $(\G,T)$ is continuous if it is \textbf{locally constant}, i.e.,
$T$ is the disjoint union of finitely many clopen subsets $T_i$
and for each $i$ there is a subgroup $A_i$ of $G$ such that
$G_t = A_i$ for every $t \in T_i$.

\begin{lem}[{\cite[Lemma 5.2.1]{R 2017}}]\label{continuous family}
Let $G$ be a profinite group and
let $\{G_t \suchthat t\in T\}$ be a collection of subgroups
indexed by a profinite space $T$.
Then the following conditions are equivalent:

\begin{enumerate}
\item[(a)] $\{G_t \suchthat t\in T\}$ is continuous;

\item[(b)] The set $\hat\G=\{(g,t)\in G\times T \suchthat t\in T, g\in G_t\}$ is closed in $G\times T$;

\item[(c)] The map $\varphi \colon T\to \Subgr(G)$, given by $\varphi(t)=G_t$, is continuous, where $\Subgr(G)$ is endowed with the \'etale topology;

\item[(d)] $\bigcup_{t\in T} G_t$ is closed in $G$.
\end{enumerate}
\end{lem}


The set $\hat\G$ in (b),
together with the projection $\pi\colon\hat\G \to T$ on the first coordinate,
is a \textbf{sheaf} of profinite groups.
Given a profinite group $H$, a \textbf{sheaf morphism} $\alpha \colon \hat\G \to H$
is a continuous map such that the restriction of $\alpha$ to $\hat\G(t) = \pi^{-1}(t)$ is a group homomorphism, for every $t \in T$.
For instance, the \textbf{inclusion} map $\hat\G \to G$, given by
$(g,t) \mapsto g$, is a sheaf morphism.
These notions are instrumental in the construction of free profinite products
(\cite[Section 5.1]{R 2017}).

By abuse of notation
we write simply $\G$ instead of $(\G,T)$ and also instead of $\hat\G$.

\begin{defi}
Consider the category of \textbf{profinite pairs}
$(G,\G)$,
where $G$ is a profinite group
and $\G$ is
a continuous family of closed subgroups of $G$,
closed under the conjugation in $G$.
A \textbf{morphism}
$\varphi \colon (G,\G) \to (A,\A)$
in this category
is a homomorphism $\varphi \colon G \to A$ of profinite groups
such that
$\varphi(\G') \subseteq \A'$,
that is,
for every $\Gamma \in \G$ there is $\Delta \in \A$
such that $\varphi(\Gamma) \le \Delta$;
it is an \textbf{epimorphism}, if
$\varphi(\G') = \A'$.

An \textbf{embedding problem} for $(G,\G)$
(cf. \cite[Definition 5.1.1]{HJ} or \cite[Definition 4.1]{Haran})
is a pair of morphisms 
\begin{equation}\label{EP pairs}
\big(\varphi\colon (G,\G)\to (A,\A),\ \alpha\colon (B,\B)\to (A,\A)\big)
\end{equation}
such that $\alpha$ is an epimorphism
and for every $\Gamma \in \G$
there exists $\Delta \in \B$
and a homomorphism $\gamma_\Gamma\colon \Gamma\to \Delta$
such that
$\alpha\circ\gamma_\Gamma=\varphi|_{\Gamma}$.

We say that \eqref{EP pairs} is \textbf{finite}, if $B$ is finite.
We say that \eqref{EP pairs} is \textbf{rigid},
if $\alpha$ is \textbf{rigid},
i.e.,
$\alpha|_{\Delta}$ is injective for every $\Delta \in \B$.

A \textbf{solution} of \eqref{EP pairs} is a morphism
$\gamma\colon (G,\G)\to (B,\B)$
such that $\alpha\circ\gamma=\varphi$.

We say that $G$ is \textbf{$\G$-projective} if
every finite embedding problem \eqref{EP pairs} for $(G,\G)$
has a solution.
Equivalently (\cite[Corollary 5.1.5]{HJ}),
every finite rigid embedding problem \eqref{EP pairs} for $(G,\G)$ has a solution.
\end{defi}

We remark that \cite{HJ}
uses the term \emph{strong $\G$-projective} instead of $\G$-projective
and both \cite{HJ} and \cite{Haran}
do not assume that $\G$ is a continuous family,
only that it is \'etale compact.
We do not know whether this is the same
(\cite[Problem 2.1.11 and Proposition 2.1.8]{HJ}).

However, since we do assume here that $\G$ is indexed by a profinite space,
it would be desirable to take this space into account in
the above definition.
This is the purpose of the following two sections.

\section{Piles}

\begin{defi}\label{pile}
A \textbf{pile}
$\bfG = (G,T)$
consists of a profinite group $G$,
a profinite space $T$, 
and a continuous action of $G$ on $T$ (from the right).
Denoting by $G_t$ the $G$-stabilizer of $t$,
for every $t \in T$,
we note that
$\G = \{G_t\}_{t \in T}$
is a continuous family of closed subgroups of $G$
(\cite[Lemma 5.2.2]{R 2017})
closed under the conjugation in $G$,
such that
$G_{t^g} = G_t^g$ for all $t \in T$ and $g \in G$.

A pile
$\bfG = (G,T)$
is \textbf{finite} if
both $G$ and $T$ are finite.

A \textbf{morphism} of group piles
$\alpha \colon \bfB=(B,Y) \to \bfA=(A,X)$
consists of
a group homomorphism
$\alpha \colon B \to A$
and a continuous map
$\alpha \colon Y \to X$
such that
$\alpha(y^b) = \alpha(y)^{\alpha(b)}$
for all $y \in Y$ and $b \in B$.
This implies
$\alpha(B_y) \le A_{\alpha(y)}$
for every $y \in Y$;
in particular,
denoting $\A = \{A_x\}_{x \in X}$
and $\B = \{B_y\}_{y \in Y}$,
we have
$\alpha(\B') \subseteq \A'$.

The \textbf{kernel $\Ker \alpha$ of $\alpha$}
is the kernel of the group homomorphism $\alpha \colon B \to A$.

The above morphism $\alpha$ is an \textbf{epimorphism} if
$\alpha(B) = A$, $\alpha(Y) = X$,
and for every $x \in X$ there is $y \in Y$ such that
$\alpha(y) = x$ and $\alpha(B_y) = A_x$.
(Then $\alpha(\B') = \A'$.)
It is \textbf{rigid}, if
$\alpha$ maps $B_y$ isomorphically onto $A_{\alpha(y)}$,
for all $y \in Y$,
and the induced map of the orbit spaces $Y/B \to X/A$ is a homeomorphism.
\end{defi}

\begin{rem}\label{rigid equivalent}
Let 
$\alpha \colon \bfB=(B,Y) \to \bfA=(A,X)$
be a morphism 
and let $K$ be its kernel.
The quotient map
$\pi \colon \bfB \to \bfB/K := (B/K,Y/K)$
is an epimorphism of piles
and 
there is a unique morphism
$\bar\alpha \colon \bfB/K \to \bfA$
such that
$\alpha = \bar\alpha \circ \pi$.

Moreover,
$\alpha$ is a rigid epimorphism
if and only if 
$\bar\alpha$ is an isomorphism
and
$K \cap B_y = 1$ for every $y \in Y$.
%
%
%
%
\end{rem}

\begin{constr}\label{standard}
Let $G$ be a profinite group and
$\{G_t \suchthat t \in T_0\}$
a continuous family of subgroups of $G$.
We construct the
\textbf{standard $G$-extension} $T$ of $T_0$
such that $(G,T)$ is a pile,
$T_0$ is a set of representatives of the $G$-orbits in $T$,
and $G_t$ is the $G$-stabilizer of $t$,
for every $t \in T_0$.

Let 
$T = \{(t, G_t g) \suchthat t \in T_0, G_t g \in G/G_t\}$
and let
$G$ act on $T$ by
$(t,G_t g)^\sigma = (t, G_t g\sigma)$.
Then
$\Stab_G(t,G_t g) =
\{\sigma \in G \suchthat G_t g\sigma = G_t g)\} = G_t^g$.

Identifying $t \in T_0$ with $(t,G_s 1) \in T$
we may view $T_0$ as a subset of $T$
such that
$T_0$ is a set of representatives of the $G$-orbits in $T$.
Then
$\Stab_G(t) = G_t$, for every $t \in T_0$.

If $G$ and $T_0$ are finite,
and we regard $T$ as a discrete space,
then the above map and the action are continuous,
and hence $(G,T)$ is a finite pile.

In the general case
we view $T$ as the quotient space of the profinite space $T_0 \times G$
via the map $\pi \colon T_0 \times G \to T$
given by
$(t,g) \mapsto (t, G_t g)$.
By \cite[Proposition 5.2.3]{R 2017}
this is a profinite space.
The $G$-action on $T$
is induced via $\pi$
from the continuous $G$-action on $T_0 \times G$
by multiplying the second coordinate from the right,
and therefore is continuous.

Hence $(G,T)$ is a pile. 
The embedding $T_0 \to T$ is also continuous,
and hence $T_0$ is a closed subset of $T$.

Moreover, the above construction is functorial in the following sense.
Let $H$ be another profinite group,
with a continuous family of subgroups
$\Hc=\{H_s \suchthat s \in S_0\}$,
and let
$\varphi \colon G \to H$ be a homomorphism
and
$\varphi' \colon T_0 \to S_0$ a continuous map,
such that $\varphi(G_t) \le H_{\varphi'(t)}$
for every $t \in T_0$.
Let $S$ be the standard $H$-extension of $S_0$.
Then
$\varphi, \varphi'$ induce a continuous map 
$T \to S$, namely,
$(t, G_t g) \mapsto (\varphi'(t), H_{\varphi'(t)} \varphi(g))$,
which, together with $\varphi$,
form a morphism of piles
$(G,T) \to (H,S)$.
\end{constr}

Later we shall need the following, easily verified, lemma:

\begin{lem}\label{connecting}
Let $\bfG = (G,T)$ be a pile
and $\bfA = (A,X)$, $\bfB = (B,Y)$ finite piles.
Let $\varphi \colon \bfG \to \bfA$ be a morphism
and $\psi\colon \bfG \to \bfB$ an epimorphism.
Assume that $\Ker(\psi) \le \Ker(\varphi)$
and the partition
$\{\psi^{-1}(y) \suchthat y \in Y\}$
is finer than
$\{\varphi^{-1}(x) \suchthat x \in X\}$.
Then there is a morphism
$\alpha \colon \bfB \to \bfA$
such that
$\alpha \circ \psi = \varphi$.
\end{lem}

\begin{lem}\label{decomposition 0}
Let $\bfG = (G,T)$ be a pile and
let $N_0$ be an open normal subgroup of $G$.
\begin{itemize}
\item[(a)]
Let $X$ be a partition of $T$.
Then there is a finite pile
$\bfB = (B,Y)$
and an epimorphism
$\psi \colon \bfG \to \bfB$
such that $\Ker(\psi) \le N_0$
and the partition
$\{\psi^{-1}(y) \suchthat y \in Y\}$
is finer than $X$.
\item[(b)]
Let $\varphi \colon \bfG \to \bfA$
be a morphism into a finite pile.
Then there is a finite pile $\bfB$,
an epimorphism
$\psi \colon \bfG \to \bfB$
such that $\Ker(\psi) \le N_0$,
and a morphism
$\alpha \colon \bfB \to \bfA$
such that $\alpha \circ \psi = \varphi$.
\end{itemize}
\end{lem}

\begin{proof}
(a)
If $N$ is an open normal subgroup of $G$,
then $(G/N, T/N)$ is a pile
and the pair of quotient maps
$(G\to G/N, T \to T/N)$ is a rigid epimorphism
$\bfG \to (G/N, T/N)$.

As $T$ is the inverse limit of $T/N$,
where $N$ runs through the open normal subgroups of $G$,
there is an open $N \normal G$ such that 
$N \le N_0$ and 
the map $T \to X$ factors through $T/N$
(\cite[Lemma 1.1.16(b)]{HJ}).
Thus, replacing $\bfG$ by $(G/N, T/N)$,
we may assume that $G$ is finite and $N_0 = 1$.

Put $B = G$
and let $\psi \colon G \to B$ be the identity.
By Lemma \ref{G-partition}
there is a $G$-partition
$Y = \{T_1, \ldots, T_n\}$ of $T$
finer than $X$,
and for every $i$ there is $t_i \in T_i$ such that
$\Stab_{G}(T_i) = \Stab_G(t_i)$.

Then 
$\bfB = (G,Y)$ is a finite pile
and the identity of $G$ together with the map $T \to Y$
define an epimorphism $\psi \colon \bfG \to \bfB$
with the required properties.

(b)
Write $\bfA = (A,X)$.
By (a)
there is a finite pile
$\bfB = (B,Y)$
and an epimorphism
$\psi \colon \bfG \to \bfB$
such that $\Ker(\psi) \le N_0$
and the partition
$\{\psi^{-1}(y) \suchthat y \in Y\}$
is finer than the partition $\{\varphi^{-1}(x) \suchthat x \in X\}$.
By Lemma~\ref{connecting}
there is a morphism $\alpha \colon B \to A$
such that
$\alpha \circ \psi = \varphi$.
\end{proof}

\begin{lem}\label{decomposition}
Let $\bfG = (G,T)$ and $\bfA = (A,X)$ be piles, $\bfA$ finite.
Let $\varphi \colon \bfG \to \bfA$ be a morphism.
Assume that 
$G_{t_1} \cap G_{t_2} = 1$
for all distinct $t_1,t_2 \in T$.
Then there is a finite pile $\bfAhat = (\Ahat,\Xhat)$
and a factorization of $\varphi$ into 
an epimorphism
$\varphihat \colon \bfG \to \bfAhat$
and a morphism
$\varphi_0 \colon \bfAhat \to \bfA$
such that
if $\xhat_1,\xhat_2 \in \Xhat$ satisfy
$\varphi_0(\xhat_1) \ne \varphi_0(\xhat_2)$,
then
$\Ahat_{\xhat_1} \cap \Ahat_{\xhat_2} \le \Ker \varphi_0$.
\end{lem}

\begin{proof}
Let $R = \{(t_1,t_2) \suchthat \varphi(t_1) \ne \varphi(t_2)\}$.
This is a clopen subset of $T \times T$, since $X$ is finite.
If $(t_1,t_2) \in R$, then 
$t_1 \ne t_2$,
and hence $G_{t_1} \cap G_{t_2} = 1$.
Therefore there is an open $N \normal G$ such that
\begin{equation}\label{intersection in ker}
G_{t_1}N \cap G_{t_2}N \le \Ker \varphi.
\end{equation}
If $(t'_1,t'_2) \in T \times T$ is sufficiently close to $(t_1,t_2)$, then 
$G_{t'_i}N \le G_{t_i}N$, for $i=1,2$,
hence
$G_{t'_1}N \cap G_{t'_2}N \le \Ker \varphi$
as well.

As $R$ is compact, there is an open $N \normal G$
such that \eqref{intersection in ker} holds
simultaneously for all $(t_1,t_2) \in R$.

By Lemma~\ref{decomposition 0}(b)
there is a finite pile $\bfAhat = (\Ahat,\Xhat)$,
an epimorphism $\varphihat \colon \bfG \to \bfAhat$,
and a morphism $\varphi_0 \colon \bfAhat \to \bfA$
such that
$\varphi_0 \circ \varphihat = \varphi$
and
$\Ker \varphihat \le N$.

Let $\xhat_1,\xhat_2 \in \Xhat$ such that
$\varphi_0(\xhat_1) \ne \varphi_0(\xhat_2)$.
Then
there are $t_1,t_2 \in T$ such that
$\varphihat(t_i) = \xhat_i$
and
$\varphihat(G_{t_i}) = \Ahat_{\xhat_i}$,
for $i = 1,2$.
Then $(t_1,t_2) \in R$,
hence \eqref{intersection in ker} holds.
We have $\Ker \varphihat \le N, G_{t_1} N, G_{t_2} N$, hence
\begin{multline*}
\Ahat_{\xhat_1} \cap \Ahat_{\xhat_2} =
\varphihat(G_{t_1}) \cap \varphihat(G_{t_2}) \le
\varphihat(G_{t_1} N) \cap \varphihat(G_{t_2} N) =
\\
\varphihat(G_{t_1} N \cap G_{t_2} N) \le
\varphihat(\Ker \varphi) = \Ker \varphi_0.
\end{multline*}
\end{proof}

\begin{defi}\label{def cartesian square}
A commutative diagram of piles
\begin{equation}\label{cartesian diagram}
\xymatrix{
\bfBhat \ar[r]^\alphahat \ar[d]^p
& \bfAhat \ar[d]_{\varphi_0}
\\
\bfB \ar[r]^\alpha & \bfA
}
\qquad
\qquad
\qquad
\xymatrix{
\bfBhat = (\Bhat,\Yhat), \ \bfAhat = (\Ahat,\Xhat),
\\
\bfB = (B,Y), \ \bfA = (A,X)
}
\end{equation}
is called a \textbf{cartesian square}
if, up to an isomorphism,
$\bfBhat = \bfB \times_{\bfA} \bfAhat$,
that is,
$\Bhat = B \times_{A} \Ahat$,\
$\Yhat = Y \times_{X} \Xhat$,
and $p,\alphahat$ are the coordinate projections.
\end{defi}

\begin{lem}\label{cartesian rigid}
Let \eqref{cartesian diagram} be a cartesian diagram.
If $\alpha$ is a rigid epimorphism,
then so is $\alphahat$.
\end{lem}

\begin{proof}

Let $(y,\xhat) \in \Bhat$. Then
\begin{multline*}
\Stab_{\Bhat}(y,\xhat) =
\{(b,\ahat) \in \Bhat \suchthat (y^b, \xhat^{\ahat}) = (y,\xhat)\} =
\\
\{(b,\ahat) \in \Bhat \suchthat b \in B_y,\, \ahat \in \Ahat_{\xhat}\} =
B_y \times_{A_x} \Ahat_{\xhat},
\end{multline*}
where $x = \alpha(y) = \varphi_0(\xhat)$.
As $\alpha$ maps $B_y$ isomorphically onto $A_x$, \
$\alphahat$ maps $B_y \times_{A_x} \Ahat_{\xhat}$ isomorphically onto $\Ahat_{\xhat}$.

As $\alpha \colon Y \to X$ is surjective,
so is $\alphahat \colon \Yhat \to \Xhat$,
and therefore also the induced map
$\Yhat/\Bhat \to \Xhat/\Ahat$.
We have to show that it is injective,
i.e., that
$(y, \xhat), (y', \xhat') \in Y \times_{X} \Xhat$,
such that $\xhat' = \xhat^{\ahat}$ for some $\ahat \in \Ahat$,
are in the same $\Bhat$-orbit.

Replacing $(y', \xhat')$ by $(y', \xhat')^{{\bhat}^{-1}}$,
where $\alphahat(\bhat) = \ahat$,
we may assume that $\xhat' = \xhat$.
Put $x = \varphi_0(\xhat)$,
then $\alpha(y) = \alpha(y') = x$.
As $Y/B \to X/A$ is a bijection,
there is $b \in B$ such that $y' = y^b$.
Apply $\alpha$ to get that
$x = x^{\alpha(b)}$.
Therefore $\alpha(b) \in A_x = \alpha(B_y)$,
whence $b = \beta \kappa$,
where $\beta \in B_y$ and $\kappa \in \Ker(\alpha)$.

It follows that $(y', \xhat') = (y^{\beta \kappa},\xhat) =
(y^{\kappa},\xhat) = (y,\xhat)^{(\kappa,1)}$.
\end{proof}


\section{Projective piles}

\begin{defi}\label{proj pile}
An \textbf{embedding problem} for a pile $\bfG$ is a pair
\begin{equation}\label{EP}
(\varphi\colon\bfG\to\bfA,\ \alpha\colon\bfB\to\bfA)
\end{equation}
of morphisms of group piles
such that $\alpha$ is a rigid epimorphism.
It is \textbf{finite}, if $\bfB$ is finite.

A \textbf{solution} of \eqref{EP} is a
morphism $\gamma\colon \bfG\to \bfB$
such that $\alpha\circ\gamma=\varphi$.

A pile $\bfG$ is \textbf{projective},
if every finite embedding problem for $\bfG$
has a solution.
\end{defi}

\begin{example}\label{basic pile}
Let $\C$ be a family of finite groups 
closed under quotients, subgroups, and extensions.
Let $\{G_t\}_{t \in T_0}$
be a finite family of finite $\C$-groups
and $F$ a finitely generated free pro-$\C$ group.
Form the free pro-$\C$ product
$G = F \amalg (\coprod_{t \in T_0} G_t)$
and let 
$T$ be the standard $G$-extension of $T_0$.
It is easy to see that
$\bfG = (G,T)$ is a projective pile.

We call a pile of this type
a \textbf{basic pro-$\C$ pile}.
\end{example}

\begin{lem}\label{complete cartesian}
Let $\alphahat \colon \bfBhat \to \bfAhat$ be a rigid epimorphism.
Assume that $\Ker \alphahat$ is a finite group.
Then there is a cartesian square \eqref{cartesian diagram}
in which $\alpha \colon \bfB \to \bfA$ 
is a rigid epimorphism of finite piles.
\end{lem}

\begin{proof}
Write $\bfBhat = (\Bhat, \Yhat)$ and let $K = \Ker \alphahat$.
By Remark~\ref{rigid equivalent} we have
$(\bigcup_{\yhat \in \Yhat}\Bhat_{\yhat}) \cap (K\smallsetminus \{1\})
= \emptyset$
and
we may assume that $\bfAhat = \bfBhat/K$
and
$\alphahat$ is the quotient map.
By 
\cite[Lemma 5.2.1]{R 2017},
$\bigcup_{\yhat \in \Yhat}\Bhat_{\yhat}$
is a closed subset of $\Bhat$.
As $K$ is finite, there is an open $N \normal \Bhat$
such that 
$K \cap N = 1$
and
$(\bigcup_{\yhat \in \Yhat}\Bhat_{\yhat}) \cap (K\smallsetminus \{1\})
= \emptyset$,
whence
$\Bhat_{\yhat}N \cap KN = N$,
for every $\yhat \in \Yhat$.
For every such $N$ the diagram
\begin{equation*}
\xymatrix{
\bfBhat \ar[r]^\alphahat \ar[d]
& \bfBhat/K \ar[d]\rlap{$=\bfAhat $}
\\
\bfBhat/N \ar[r] & \bfBhat/\alphahat(N)
}
\end{equation*}
is cartesian and
by Remark~\ref{rigid equivalent}
the bottom map is a rigid epimorphism.
Since a composition of cartesian diagrams is again a cartesian diagram,
we may replace
$\bfBhat$ by $\bfBhat/N $ to assume that $\Bhat$ is a finite group.

As $\alphahat$ is rigid,
$\Stab_K(\yhat) = 1$ for every $\yhat \in \Yhat$.
Thus, by Lemma~\ref{G-partition},
applied to the pile $(K,\Yhat)$,
there is a partition $Y = \{\Yhat_i\}_{i=1}^n$ of $\Yhat$
such that
$\Stab_K(\Yhat_i) = 1$ 
for every $i$.
Replacing $Y$ by a refinement we may assume that $Y$ is a $\Bhat$-partition.
Thus $\Yhat \to Y$, together with the identity of $\Bhat$,
induces a morphism of piles
$p \colon \bfBhat \to \bfB := (\Bhat, Y)$.
Then
\begin{equation*}
\xymatrix{
\bfBhat \ar[r]^\alphahat \ar[d]_{p} & \bfBhat/K \ar[d]
\\
\bfB \ar[r] & \bfB/K
}
\end{equation*}
is a cartesian square,
because
$\Stab_K(\Yhat_i) = 1$ implies that 
$\Yhat_i \cap \Yhat_i^\kappa = \emptyset$
for every $1 \ne \kappa \in K$,
hence $\Yhat_i$ contains at most one element of 
every $K$-orbit in $\Yhat$.
Moreover,
$K \cap \Stab_{\Bhat}(\Yhat_i) = \Stab_K(\Yhat_i) = 1$,
hence the bottom map is a rigid epimorphism.
\end{proof}

\begin{prop}\label{general EP}
Let $\bfG$ be a projective pile.
Then every embedding problem
(not necessarily finite)
for $\bfG$ has a solution.
\end{prop}

\begin{proof}
Let
$(\varphihat \colon \bfG \to \bfAhat,\ \alphahat \colon \bfBhat \to \bfAhat)$
be an embedding problem for $\bfG = (G,T)$.
We may assume that $\alphahat$ is the quotient map
$\bfBhat \to \bfBhat/K$
for some $K \normal G$.

If $K$ is finite, then by Lemma~\ref{complete cartesian}
we have a commutative diagram
\begin{equation}\label{EP with cartesian diagram}
\xymatrix{
& \bfG \ar[d]_\varphihat \ar@/^15pt/[dd]^\varphi
\\
\bfBhat \ar[r]^\alphahat \ar[d]^p
& \bfAhat \ar[d]_{\varphi_0}
\\
\bfB \ar[r]^\alpha & \bfA
}
\end{equation}
with a cartesian square,
in which $\alpha$ is a rigid epimorphism of finite piles.
By assumption there is a morphism
$\gamma \colon \bfG \to \bfB$
such that
$\alpha \circ \gamma = \varphi = \varphi_0 \circ \varphihat$.
By the universal property of fibred products
there is a unique morphism
$\gammahat \colon \bfG \to \bfBhat$
such that
$\alphahat \circ \gammahat = \varphihat$
(and $p \circ \gammahat = \gamma$).
Thus $\gammahat$ is a solution of the given embedding problem.

The general case is verbally identical with
Part II in the proof of \cite[Lemma 7.3]{HJ-real}.
\end{proof}

\begin{lem}\label{completion of solution}
Let 
$\varphi\colon\bfG\to\bfA$ and $\alpha\colon\bfB\to\bfA$
be morphisms of piles,
$\bfB=(B,Y)$ and $\bfA=(A,X)$ finite.
Let $\psi \colon G \to B$ be a group homomorphism 
such that $\alpha \circ \psi = \varphi$.
Then $\psi$ can be completed to 
a morphism
$\psi \colon \bfG\to \bfB$
such that $\alpha \circ \psi = \varphi$
if and only if
for every $t \in T$ there is $y \in Y$
such that
\begin{equation}\label{basic}
\alpha(y) = \varphi(t)
\text{ and }
\psi(G_t) \le B_y.
\end{equation}
\end{lem}

\begin{proof}
If $\psi$ extends to $\bfG \to \bfA$,
then \eqref{basic} holds with $y = \psi(t)$.

Conversely, assume that the condition holds.
Let $t \in T$ and fix $y \in Y$ such that satisfies \eqref{basic}.
As $\varphi \colon T \to X$ is continuous,
there is a clopen neighborhood $U_t$ of $t$ in $T$ such that
$\varphi(U_t) = \{\alpha(y)\}$.
As $\psi(G_t) \le B_y$,
by Lemma~\ref{Stab}(c),(f) we may assume that
$\psi(\Stab_G(U_t)) \le B_y$.

As $T$ is compact,
$\{U_t\}_{t \in T}$
has a finite subcovering.
Therefore there is a partition $T = \bigdotcup_{i=1}^n T_i$
and there are $y_1,\ldots, y_n \in Y$
such that
$\varphi(T_i) = \{\alpha(y_i)\}$
and
$\psi(\Stab_G(T_i)) \le B_{y_i}$,
for every $i$.

If we now define $\psi \colon T \to Y$
by mapping $T_i$ onto $y_i$,
then $\psi$ is continuous and
$\alpha\circ\psi=\varphi$ on $T$.

By Lemma~\ref{Stab}(a)
we may assume that $\{T_i\}_{i=1}^n$ is a $G$-partition.
Without loss of generality,
if $T_j = T_i^g$, with $g \in G$,
then $y_j = y_i^{\psi(g)}$.
Indeed, we fix a representative $T_i$ of a $G$-orbit in 
$\{T_i\}_{i=1}^n$ 
and for $T_j = T_i^g$ redefine $y_j$ to be $y_j = y_i^{\psi(g)}$.
This definition is good:
If $T_i^g = T_i^h$,
then $hg^{-1} \in \Stab_G(T_i)$,
hence
$\psi(h) \psi(g)^{-1} = \psi(hg^{-1}) \in B_{y_i} = \Stab_B(y_i)$,
whence
$y_i^{\psi(g)} = y_i^{\psi(h)}$.

It then follows that
$\psi(t^g) = \psi(t)^{\psi(g)}$
for every $g \in G$.
%
\end{proof}

\begin{cor}\label{completion of homomorphism}
Let $\bfG=(G,T)$ be a pile,
$\bfB=(B,Y)$ a finite pile,
and $\psi \colon G \to B$ a group homomorphism.
Assume that
for every $t \in T$ there is $y \in Y$
such that
$\psi(G_t) \le B_y$.
Then $\psi$ can be completed to a morphism
$\psi \colon \bfG\to \bfB$.

Moreover,
let $\mathcal{P}$ be a partition of $T$,
and put
$\B = \{B_y \suchthat y \in Y\}$.
Then there is $N \in \N$ such that
if
\begin{equation}\label{Y large}
|\{y \in Y \suchthat B_y = B'\}| \ge N
\text{ for all }
B' \in \B,
\end{equation}
then $\psi$ can be completed so that
the partition $\{\psi^{-1}(\{y\}) \suchthat y \in Y\}$
is finer than $\mathcal{P}$.
\end{cor}

\begin{proof}
Let $\bfA = (1, \{*\})$ be the trivial pile
and let 
$\varphi\colon\bfG\to\bfA$ and $\alpha\colon\bfB\to\bfA$
be the unique morphisms of piles.
Then the group homomorphisms satisfy
$\alpha \circ \psi = \varphi$,
and
$\alpha(y) = * = \varphi(t)$
for all $t \in T$ and $y \in Y$.
By Lemma~\ref{completion of solution},
$\psi$ can be completed to a morphism of piles.

Moreover, given a partition $\mathcal{P}$ of $T$,
we may replace $\{T_i\}_{i=1}^n$,
in the proof of Lemma~\ref{completion of solution}
for this particular case,
by a partition finer than $\mathcal{P}$.
Then, if \eqref{Y large} holds for a sufficiently large $N$,
we can choose $y_1,\ldots, y_n$ above
from distinct $B$-orbits in $Y$.
Then
$y_1,\ldots, y_n$ remain distinct even after we replace them by their conjugates,
and hence
$\{\psi^{-1}(\{y\}) \suchthat y \in Y\}
=
\{\psi^{-1}(\{y_i\})\}_{i=1}^n = 
\{T_i\}_{i=1}^n$.
\end{proof}

\begin{prop}\label{pile vs group projectivity}
A pile $\bfG = (G,T)$ is projective
if and only if
\begin{itemize}
\item[(a)]
$G$ is $\{G_t \suchthat t \in T\}$-projective;
and
\item[(b)]
if $t, t' \in T$ are distinct, then
$G_t \cap G_{t'} = 1$.
\end{itemize}
\end{prop}

\begin{proof}
Assume that $\bfG$ is projective.
Put $\G = \{G_t \suchthat t \in T\}$.

(a)
Consider a finite rigid embedding problem 
\begin{equation*}\tag{\ref{EP pairs}}
\big(\varphi\colon (G,\G)\to (A,\A),\ \alpha\colon (B,\B)\to (A,\A)\big)
\end{equation*}
for $(G, \G)$.

Write $\B$ as $\{B_y \suchthat y \in Y_0\}$
and for every $y \in Y_0$ let $A_y = \alpha(B_y)$.
Then $\A = \{A_y \suchthat y \in Y_0\}$, 
and $\A = \A^A$, $\B = \B^B$.
Let $Y$ be the standard $B$-extension of $Y_0$ with respect to $\B$,
and let $X$ be the standard $A$-extension of $Y_0$ with respect to $\A$
(Construction \ref{standard}).
Then $\bfB = (B,Y)$ and $\bfA = (A,X)$ are piles
and the identity map $Y_0 \to Y_0$ extends $\alpha$ to a rigid epimorphism $\alpha \colon \bfB \to \bfA$.

As $\varphi$ is a morphism of pairs,
for every $t \in T$ there is $x \in X$
such that $\varphi(G_t) \le A_x$.
Hence by Corollary~\ref{completion of homomorphism},
$\varphi$ can be completed to a morphism
$\varphi \colon \bfG \to \bfA$.

As $\bfG$ is projective, there is a morphism
$\psi \colon \bfG \to \bfB$ such that ${\alpha \circ \psi} = \varphi$.
In particular, the group homomorphism
$\psi \colon G \to B$ satisfies $\alpha \circ \psi = \varphi$
and $\psi(G_t) \le B_{\psi(t)}$, for every $t \in T$,
that is,
$\psi(\G) \subseteq \B'$.
Thus $\psi$ solves \eqref{EP pairs}.

(b)
Let $t, t' \in T$ be distinct.
It suffices to show that
$G_t \cap G_{t'}$ is contained in every open subgroup $N$ of $G$.

By Lemma~\ref{decomposition 0}(a)
there is an epimorphism
$\varphi \colon \bfG \to \bfA$
onto a finite pile $\bfA = (A,X)$
such that
$\Ker(\varphi) \le N$
and $\varphi(t), \varphi(t') \in X$ are distinct.

Let $Y_0$ be a set of representatives of the $A$-orbits of $X$.
Let $B = A \amalg(\amalg_{y \in Y_0} A_y)$
and let $\alpha \colon B \to A$ be the epimorphism
that maps $A, A_y$ identically on the corresponding subgroups of $A$.
Put $\B = \{A_y^b \suchthat y \in Y_0,\, b \in B\}$.
Let $Y$ be the standard $B$-extension of $Y_0$
with respect to $\B$,
so that $\bfB = (B,Y)$ is a pile.
By Construction~\ref{standard},
$\alpha$ together with the identity of $Y_0$
extend to a morphism $\alpha \colon \bfB \to \bfA$.
It is easy to see that $\alpha$ is a rigid epimorphism.

By Proposition~\ref{general EP}
there is a morphism $\psi \colon \bfG \to \bfB$
such that $\alpha \circ \psi = \varphi$.
Then $\psi(t) \ne \psi(t')$,
because $\alpha \circ \psi(t) = \varphi(t) \ne \varphi(t') = \alpha \circ \psi(t')$.
By \cite[Lemma 3.1.10]{HJ}, 
$B_{\psi(t)} \cap B_{\psi(t')} = 1$.
Thus
$\varphi(G_t \cap G_{t'}) =
\alpha(\psi(G_t \cap G_{t'})) \le
\alpha ( \psi(G_t) \cap \psi(G_{t'}) ) \le
\alpha(B_{\psi(t)} \cap B_{\psi(t')})
= 1
$,
whence
$G_t \cap G_{t'} \le \Ker(\varphi) \le N$.

Conversely, assume that (a) and (b) hold.
Let \eqref{EP} be a finite embedding problem for $\bfG$,
with $\bfA = (A,X)$ and $\bfB = (B,Y)$.

Let $\bfAhat = (\Ahat,\Xhat)$
and $\varphihat, \varphi_0$ be as in Lemma~\ref{decomposition}.
Then there is a commutative diagram
with a cartesian square
\begin{equation*}\tag{\ref{EP with cartesian diagram}}
\xymatrix{
& \bfG \ar[d]_\varphihat \ar@/^15pt/[dd]^\varphi
\\
\bfBhat \ar[r]^\alphahat \ar[d]^p
& \bfAhat \ar[d]_{\varphi_0}
\\
\bfB \ar[r]^\alpha & \bfA
}
\end{equation*}
By Lemma~\ref{cartesian rigid},
$(\varphihat, \alphahat)$
is also a finite embedding problem for $\bfG$.

As $G$ is $\G$-projective,
there is a group homomorphism $\psihat \colon G \to \Bhat$
such that $\alphahat \circ \psihat = \varphihat$
and for every $t \in T$ there is
$\yhat \in \Yhat$
such that
\begin{equation}\label{hatcontainment}
\psihat(G_t) \le \Bhat_{\yhat}.
\end{equation}
Put $\psi = p \circ \psihat$
and $y = p(\yhat) \in Y$.
Then
$\alpha \circ \psi = \varphi$
and $\psi(G_t) \le B_y$.

If $\varphi(G_t) = 1$,
use $\alpha(Y) = X$ to choose $y' \in Y$ such that
$\alpha(y) = \varphi(y')$.
We have $\alpha(\psi(G_t)) = \varphi(G_t) = 1$,
and $\alpha$ is injective on $B_y$, hence also on $\psi(G_t)$,
so $\psi(G_t) = 1 \le B_{y'}$.
Thus condition \eqref{basic} of Lemma~\ref{completion of solution} holds
with $y'$ instead of $y$.

If $\varphi(G_t) \ne 1$, then
$\varphihat(G_t) \not\le \Ker(\varphi_0)$.
But, by \eqref{hatcontainment},
$$
\varphihat(G_t) = \alphahat \circ \psihat (G_t) \le
\alphahat(\Bhat_{\yhat}) = A_{\alphahat(\yhat)}
$$
and $\varphihat(G_t) \le \Ahat_{\varphihat(t)}$, hence
$\Ahat_{\alphahat(\yhat)} \cap \Ahat_{\varphihat(t)}
\not \le \Ker(\varphi_0)$.
By Lemma \ref{decomposition},
$\varphi_0(\alphahat(\yhat)) = \varphi_0(\varphihat(t))$,
that is,
$\alpha(y) = \varphi(t)$.
Thus condition \eqref{basic} of Lemma~\ref{completion of solution} holds.

By Lemma~\ref{completion of solution},
$\psi$ can be completed to a solution
of \eqref{EP}.
\end{proof}

\begin{rem}\label{pile vs group projectivity b}
If $G$ is a $\G$-projective group,
then $\Gamma \cap \Gamma' = 1$
for all distinct $\Gamma, \Gamma' \in \G$
(\cite[Proposition~5.5.3(b)]{HJ}).
Hence we can replace (b)
in Proposition~\ref{pile vs group projectivity}
by
\begin{itemize}
\item[(b')]
the map $t \mapsto G_t$
is injective on $T' = \{t \in T \suchthat G_t \ne 1\}$.
\end{itemize}
\end{rem}

\begin{lem}\label{mod tilde}
Let $\bfG = (G,T)$ be a projective pile.
Let $N$ be a normal subgroup of $G$ and
$\tN=\langle N\cap G_t \suchthat t\in T\rangle$.
Then the quotient pile $\bfG/\tN := (G/\tN, T/\tN)$
is projective.
\end{lem}

\begin{proof}
Let $\pi \colon \bfG \to \bfG/N$ be the quotient map.
We have to solve a finite embedding problem
$(\varphi_N, \alpha)$
in the following diagram
\begin{equation}\label{quotientEP}
\xymatrix{
& \bfG \ar[d]_{\pi}
\\
& \bfG/\tN \ar[d]_{\varphi_N}
\\
\bfB \ar[r]^\alpha
& \bfA \rlap{ $.$}
}
\end{equation}
Then
$(\varphi_N \circ \pi, \alpha)$
is a finite embedding problem for $\bfG$.
Since $\bfG$ is projective,
this embedding problem has a solution $\gamma$.

Note that $N\cap G_t\leq \Ker(\gamma_{|G_t})\leq \Ker(\gamma)$
(as $(\varphi_N \circ \pi)(N \cap G_t) = 1$
and $\alpha_{|\gamma(N \cap G_t)}$ is injective),
for every $t \in T$.
Hence $\gamma(\tN) = 1$,
whence $\gamma$ factors via $\pi$,
and so we obtain a solution of $(\varphi_N, \alpha)$.
\end{proof}

Since projective pro-$p$ groups are free pro-$p$
(\cite[Theorem 7.7.4]{RZ}),
we get:

\begin{cor}\label{projectivity mod tilde}
In the above setting
$N/\tN$ is a projective profinite group.
If $G$ is pro-$p$, then 
$N/\tN$ is free pro-$p$.
\end{cor}

\begin{cor}\label{kill all stabilizers}
In the above setting
$G/\langle G_t\mid t\in T\rangle$
is a projective profinite group.
If $G$ is pro-$p$, then it is free pro-$p$.
\end{cor}

\begin{proof}
Let $N = \langle G_t\mid t\in T\rangle$.
Then 
$\tN=\langle N\cap G_t \suchthat t\in T\rangle = N$.
By Lemma \ref{mod tilde},
$\bfG/N$ is projective.
By Proposition \ref{pile vs group projectivity},
$G$ is $\{1\}$-projective,
i.e., projective.
\end{proof}

\begin{lem}\label{sub free prod}
Let $H$ be a free product of a continuous family $\Hc$
of its subgroups
and let $G$ be a closed subgroup of $H$.
Put $\G = \{G \cap \Delta^h \suchthat \Delta \in \Hc,\ h \in H\}$.
Then there is a profinite $G$-space $T$ such that
\begin{itemize}
\item [(a)]
$\G = \{G_t := \Stab_G(t) \suchthat t \in T\}$;
\item [(b)]
if $t, t' \in T$ are distinct, then
$G_t \cap G_{t'} = 1$;
\item [(c)]
the map $t \mapsto G_t$ is injective
on $\{t \in T \suchthat G_t \ne 1\}$;
\item [(d)]
$G$ is $\G$-projective.
\end{itemize}
Thus, $\bfG = (G,T)$ is a projective pile.
\end{lem}

\begin{proof}
(a)
Write $\Hc = \{H_t\}_{t \in T_0}$
and let $T$ be the standard $H$-extension of $T_0$.
Then $\Hc^H = \{H_t \suchthat t \in T\}$.
For $t \in T$ put $G_t = G \cap H_t$.
Then $\G = \{G_t \suchthat t \in T\}$.
Furthermore,
$G$, as a subgroup of $H$, acts on $T$.
By Construction~\ref{standard}(a),
$\Stab_G(t) = G \cap \Stab_H(t) = G \cap H_t = G_t$.

(b)
By \cite[Proposition 4.8.3(a)]{HJ}, 
$H_t \cap H_{t'} = 1$.
As $G_t \le H_t$ and $G_{t'} \le H_{t'}$,
also
$G_t \cap G_{t'} = 1$.

(c)
Follows from (b).

(d)
\cite[Proposition 5.4.2]{HJ}.

By Proposition~\ref{pile vs group projectivity},
$\bfG$ is projective.
\end{proof}

\begin{lem}\label{Ker projective}
Let $G$ be a $\G$-projective pro-$p$ group.
Let $\rho \colon G \to L$ be an epimorphism,
injective on every $\Gamma \in \G$.
Extend $\rho$ to an epimorphism
$\rho_L \colon G \amalg L \to L$ by the identity of $L$.
Then $\Ker \rho_L$ is a free pro-$p$ group.
\end{lem}

\begin{proof}
Let $H = G \amalg L$ and $K = \Ker(\rho_L)$.
Put $\Hc = \G \cup \{L\} \subseteq \Subgr(H)$.

Clearly, $L$ is $\{L\}$-projective.
By \cite[Lemma 5.2.1]{HJ},
$H$ is $\Hc$-projective,
and hence also $\Hc^H$-projective
(\cite[Lemma 5.2.4]{HJ}).
As $\rho_L$ is injective on every $\Gamma \in \Hc$,
and hence on every $\Gamma \in \Hc^H$,
we have 
$K \cap \Gamma = 1$ for every $\Gamma \in \Hc^H$.
By \cite[Proposition 5.4.2]{HJ},
$K$ is $\{1\}$-projective,
that is, projective.
Therefore
$K$ is free pro-$p$.
\end{proof}

\section{HNN-extensions}

Let $G$ be a pro-$\C$ group,
let $\G = \{G_t \suchthat t \in T\}$
be a continuous family of its subgroups,
and let $\phi \colon \G \to G$ be a morphism
such that
$\phi_t := \phi|_{G_t} \colon G_t \to G$ is injective,
for every $t \in T$.

A \textbf{pro-$\C$ HNN-extension}
$\tG :=\HNN(G,T,\G,\phi)$
is a special case of the fundamental pro-$\C$ group
of a profinite graph of pro-$\C$ groups $(\G, \Gamma)$
(see \cite[Example 6.2.3(e)]{R 2017}).
Namely, $\tG$ can be thought of as
$\Pi_1^{\C}(\G, \Gamma)$,
where
\begin{itemize}
\item[(a)]
$\Gamma$ is a bouquet of loops
(i.e., a profinite graph having just one vertex $v$)
with $T$ as the space of edges,
such that $T$ is closed in $\Gamma= \{v\} \cup T$;
\item[(b)]
$\G$ turns into a sheaf over $\Gamma$ by putting $G_v = G$;
\item[(c)]
the boundary maps $\partial_0, \partial_1 \colon \G \to G$
are the inclusion and $\phi$, respectively.
\end{itemize}

Thus an HNN-extension can be explicitly defined as follows.
Let $F(T)$ denote the free pro-$\C$ group on $T$
(see \cite[Section 3.3]{RZ}).

\begin{defi}\label{HNN-ext}
The \textbf{HNN-extension}
$\tG :=\HNN(G,T,\G,\phi)$
is the quotient of the free pro-$\C$ product
$G \amalg F(T)$
modulo the relations
\begin{equation}\label{HNN relations}
\phi_t(g_t) = g_t^{t},
\qquad g_t\in G_t,\ t\in T.
\end{equation}
We call $G$ the \textbf{base group},
$T$ the set of \textbf{stable letters},
and the subgroups $G_t$ and $\phi(G_t)$
\textbf{associated}.
The inclusion $G \to G \amalg F(T)$ induces a homomorphism
$\eta \colon G \to \tG$.

Obviously, $\tG$ has the following \textbf{universal property}:
Given a pro-$\C$ group $L$,
a homomorphism $\beta \colon G \to L$,
and a continuous map $\zeta \colon T \to L$,
such that
$\beta(\phi_t(g_t)) = \beta(g_t)^{\zeta(t)}$,
for all $g_t \in G_t$ and all $t \in T$,
then there exists a unique homomorphism
$\tilde\beta \colon \tG \to L$
such that
$\beta = \tilde\beta \circ \eta$.

We call the HNN-extension \textbf{special}, if
$\phi$ is the inclusion (so that every $t$ centralizes $G_t$)
and $\G$ is locally constant, that is,
there is a partition $T = \bigdotcup_{i=1}^n T_i$ of $T$
and there are subgroups $G_1,\ldots, G_n$ of $G$
such that
$\G = \bigdotcup_{i=1}^n T_i \times G_i$.
In particular, $G_t = G_i$ for every $t \in T_i$.
In this case we usually write
$\tG :=\HNN(G,T,\G)$,
omitting $\phi$.

\end{defi}

\begin{rem}\label{general HNN}
One could conceive a more general definition 
of a bouquet $\Gamma$ and an HNN-extension $\tG$.
Namely,
(a)
$\Gamma = \{v\} \dotcup T$
is a profinite graph having just one vertex $v$,
with $T$ as the space of edges,
this time
not necessarily closed in $\Gamma$;
(b)
$\G$ is a sheaf over $\Gamma$ such that $G_v = G$;
and
(c) above holds.
Let
$\phi \colon \Gamma \to G$ be a morphism
such that $\phi_v = \iid_G$ and
$\phi_t \colon G_t \to G$ is injective,
for every $t \in T$.
Then
$\tG :=\HNN(G,\Gamma,\G,\phi)$
is the quotient of the free pro-$\C$ product
$G \amalg F(\Gamma,v)$
modulo the relations ~\eqref{HNN relations},
where $F(\Gamma,v)$ is the free pro-$\C$ group
on the pointed profinite space $(\Gamma,v)$
(see \cite[Section 3.3]{RZ}).
(We could also add the relations
$\phi_v(g_v) = g_v^{v}$ for all $g_v\in G_v = G$,
but $v$, as an element of $F(\Gamma,v)$, is $1$,
and $\phi_v = \iid_G$,
so these relations are redundant.)

In this more general setting 
the definition of a special HNN-extension is the same
as in Definition~\ref{HNN-ext},
except that the $T_i$ in the partition
$T = \bigdotcup_{i=1}^n T_i$ 
are not only clopen in $T$, but also open in $\Gamma$
(or, equivalently, the $T_i \cup \{v\}$ are closed in $\Gamma$).

This definition
of a special HNN-extension is the one adopted in \cite{permhnn}
(see \cite[Definition 2.9, Remark 2.10]{permhnn}),
from which we are going to quote some results.
However, 
it turns out that it is equivalent to a special HNN-extension
of our Definition~\ref{HNN-ext}.

Indeed,
let
$T = \bigdotcup_{i=1}^n T_i$ 
with
$T_i \cup \{v\}$ closed in $\Gamma$.
If $\phi$ is the inclusion,
then \eqref{HNN relations} is equivalent to
\begin{equation*}
t^{g_t} = t,
\qquad g_t\in G_i,\ t\in T_i,\ 1 \le i \le n,
\end{equation*}
that is,
\begin{equation*}
t^{g_t} = t,
\qquad g_t\in G_i,\ t\in F(T_i,v).
\end{equation*}

Then $F(T_i \cup \{v\},v) \cong \coprod_{i=1}^n F(T_i \cup \{v\},v)$.
By Lemma~\ref{free group},
there are profinite spaces, $T'_1,\ldots, T'_n$
such that $F(T_i \cup \{v\},v) \cong F(T'_i)$ for each $i$.
Put $T' = \bigdotcup_{i=1}^n T'_i$,
then $F(\Gamma,v) \cong F(T')$.
\end{rem}

\begin{rem}\label{choice stable letters}
Let $\tG :=\HNN(G,T,\G)$ be a special HNN-extension.
The identity of $G$ extends to an epimorphism $\iota \colon \tG\to G$
that maps $T$ to the trivial element of $G$.
Then $\tG = G \ltimes \Ker(\iota)$.
In particular, the map $G \to \tG$ is injective.
On the other hand,
the identity of $F(T)$
together with the map that maps $G$ to the trivial element
extend to an epimorphism  $\tG \to F(T)=\tG/\langle G\rangle^G$,
and so the map $F(T) \to \tG$ is injective.

Moreover, if $\tG = G \ltimes \tF$ is another expression as semidirect product,
then we can choose the space of stable letters to be in $\tF$.
Indeed, for every $t\in T$ we have $t=h_t t'$,
for unique $h_t\in G$, $t'\in \tF$,
and so $T' = \{t' \suchthat t \in T\}$ is the needed space of stable letters,
homeomorphic to $T$,
with associated subgroups $G_t^{h_t}$.
Thus $\tF$ becomes the normal closure
$\tF=\langle T \rangle^{\tG}$ of $T$ in $\tG$.
\end{rem}

\begin{thm}\label{permext}\cite[Corollary 5.2]{permhnn}
Let $G = L \ltimes F$ be a semidirect product of a finite $p$-group $L$
and a free pro-$p$ group $F$.
Suppose that every torsion element of $G$ is $F$-conjugate into $L$.
Then $G$ is a special pro-$p$ HNN-extension with base group $L$.
\end{thm}



\begin{lem}\label{special properties}
Let $\tG=\HNN(L, S, \{L_s\}_{s \in S})$
be a pro-$p$ special HNN-extension
with $L$ finite.
Put $\tF = \langle S \rangle^{\tG}$ 
and let
\begin{equation*}
S' = \{s\in S \suchthat L_s \neq 1\},
\quad
S''= S \smallsetminus S',
\quad
R' = \langle C_{\tF}(L_s)) \suchthat s \in S' \rangle^{\tG}.
\end{equation*}
Then
\begin{itemize}
\item[(a)]
For
$M \le L$
we have
$C_{\tF}(M) =
\langle s^l \suchthat s \in S,\, l \in L,\, M \le L_s^l \rangle$.
\item[(b)]
$R' = \langle S' \rangle^{\tG}$.
\item[(c)]
$\tG/ \langle S' \rangle^{\tG} = L \amalg F(S'')$.
\end{itemize}
\end{lem}

\begin{proof}
(a)
By \cite[Proposition 2.11]{permhnn},
$\tF$ is the free pro-$p$ group on its subspace
$\tS = \{s^l \suchthat s \in S,\, l \in L\}$,
and the $L$-stabilizer of $s^l \in \tS$ is $L_s^l$.
Therefore by Lemma \ref{centralizer},
$C_{\tF}(M)= \langle s^l \in \tS \suchthat M \le L_s^l \rangle$.

(b)
Let $M = L_s$.
Then $M \ne 1$, hence, by (a),
$C_{\tF}(M)$ is contained in
$\langle (s')^{l'} \in \tS \suchthat
s' \in S,\, l' \in L, \, L_{s'} \ne 1\rangle
\le \langle S' \rangle^{\tG}$.
Thus $R' \le \langle S' \rangle^{\tG}$.
But $s \in C_{\tF}(L_s)$ for all $s \in S'$,
so $\langle S' \rangle^{\tG} \le R'$.

(c)
From the presentation \eqref{HNN relations} of a special HNN-extension
\begin{equation*}
\tG/ \langle S' \rangle^{\tG} =
L \amalg F(S'') \amalg F(S')/ \langle S' \rangle^{\tG} =
L \amalg F(S'').
\qedhere
\end{equation*}
\end{proof}

Apart from special HNN-extensions
we shall be using HNN-exten\-sions only of the following two kinds:

\begin{defi}\label{normal HNN}
Let $\bfG=(G,T)$ be a pile of pro-$\C$ groups,
$T_0$ a closed subset of $T$,
and $\rho\colon G\to L$ be a homomorphism of pro-$\C$ groups,
injective on $G_t$, for every $t \in T_0$.

Put $\G = \{G_t \suchthat t \in T_0\}$
and let $G_L=G\amalg L$ be the free pro-$\C$ product of $G$ and $L$.
Then $G, L \subseteq G_L$.
Hence $\G$ is also a continuous family of subgroups of $G_L$
and $\rho$ can be viewed as a sheaf morphism $\G \to G_L$.
This point of view allows us to form the
\textbf{HNN'-extension}

\begin{equation}\label{ourHNN}
\ourHNN(G,T_0,\rho,L) := \HNN(G_L,T_0,\G,\rho).
\end{equation}
Thus, explicitly,
$\ourHNN(G,T_0,\rho,L)$ is the quotient of the free pro-$\C$ product
$G \amalg L \amalg F(T_0)$ modulo the relations
\begin{equation}\label{ourHNN relations}
\rho(g_t) = g_t^{t},
\qquad g_t\in G_t,\ t\in T_0.
\end{equation}

The inclusion $G \to G \amalg L \amalg F(T_0)$ induces the \textbf{HNN-map}
$\eta_G \colon G \to \tG$.

By \eqref{ourHNN relations},
$\rho$ extends to a homomorphism
$\trho \colon \tG\to L$,
the \textbf{HNN-extension of $\rho$},
that maps $L$ identically onto itself and $T$ to $1$.
\end{defi}

If $T_0 = T$ and $g \in G$,
then each relation $\rho(g_t) = g_t^t$,
for $g_t \in G_t$,
produces the conjugate relation
$\rho(g_t)^{\rho(g)} = g_t^{t\rho(g)}$,
in addition to the relation
$\rho(g_t^g) = (g_t^g)^{t^{\trsl g}}$.
(Here and henceforth
$t^{\trsl g} \in T \subseteq G \amalg L \amalg F(T)$
is the outcome of the action of
$g \in G$ on $t \in T$,
which is a part of the definition of $\bfG$;
we reserve the notation $t^g$ to denote the conjugate $g^{-1} t g$
of $t$ in $G \amalg L \amalg F(T)$.)
This produces too many centralizers
$(g t^{\trsl g})(t \rho(g))^{-1}$ of $G_t$.
Therefore we shall give a definition of an HNN-extension of a pile,
which is an HNN'-extension with set of stable letters $T_0$,
the image of a continuous section $T/G\to T$,
if such a section exists.

\begin{defi}\label{pileHNN}
Let $\bfG=(G,T)$ be a pile of pro-$\C$ groups,
and $\rho\colon G\to L$ be a homomorphism of pro-$\C$ groups,
injective on $G_t$, for every $t \in T$.
The \textbf{pile HNN-extension}
$$
\Gbar :=\pHNN(G,T,\rho,L)
$$
is the quotient of the free pro-$\C$ product
$G \amalg L \amalg F(T)$
modulo the relations:
\begin{equation}\label{relation of pile HNN}
g^{-1}t\rho(g)=t^{\trsl g}
\text{ for all } t\in T \text{ and } g \in G.
\end{equation}
Here $t^{\trsl g} \in T \subseteq F(T) \le G \amalg L \amalg F(T)$
is the outcome of the action of
$g \in G$ on $t \in T$,
which is a part of the definition of $\bfG$;
we reserve the notation $t^g$ to denote the conjugate $g^{-1} t g$
of $t$ in $G \amalg L \amalg F(T)$.
We call $G \amalg L$ the \textbf{base group},
$T$ the set of \textbf{stable letters},
and the subgroups $G_t$ and $\rho(G_t)$
\textbf{associated}.
The inclusion $G \to G \amalg L \amalg F(T)$ induces the \textbf{HNN-map}
$\zeta_G \colon G \to \Gbar$
and $\rho$ extends to 
$\bar\rho \colon \Gbar \to L$,
the \textbf{HNN-extension of $\rho$},
that maps $L$ identically onto itself and $T$ to $1$.
\end{defi}

This construction has an obvious functorial property:

\begin{lem}\label{functorial}
Let $\bfG=(G,T), \bfG' = (G',T')$ be piles of pro-$\C$ groups
and let $\rho\colon G\to L$
and $\rho'\colon G'\to L'$
be homomorphisms of pro-$\C$ groups,
injective on every $G_t$, resp. every $G'_t$.
Put 
$\Gbar :=\pHNN(G,T,\rho,L)$
and
$\Gbar' :=\pHNN(G',T',\rho',L')$
and let $\bar\rho, \bar\rho'$
be the HNN-extensions of $\rho, \rho'$, respectively.
Let $\psi = (\psi_G, \psi_T) \colon \bfG \to \bfG'$ be a pile morphism
and let
\begin{equation}\label{HNN diagram 1}
\xymatrix{
G \ar[r]^\rho \ar[d]_{\psi_G} & L \ar[d]^{\lambda}
\\
G' \ar[r]_{\rho'} & L'
\\
}
\end{equation}
be a commutative diagram of homomorphisms of profinite groups.
Then $\psi_G, \lambda, \psi_T$ induce a homomorphism
$\psibar \colon \Gbar \to \Gbar'$ such that
\begin{equation}\label{HNN diagram 2}
\xymatrix{
G \ar[r]^{\zeta_G} \ar[d]_{\psi_G}
& \Gbar \ar[d]^{\psibar} \ar[r]^{\bar\rho}
& L \ar[d]^{\lambda}
\\
G' \ar[r]_{\zeta_{G'}}
& \Gbar' \ar[r]_{\bar\rho'}
& L'
\\
}
\end{equation}
commutes.
\end{lem}

\begin{proof}
The maps
$\psi_G \colon G \to G'$,
$\lambda \colon L \to L'$,
and
$\psi_T \colon T \to T'$
extend to a homomorphism
$\psihat \colon G \amalg L \amalg F(T) \to G' \amalg L' \amalg F(T')$.
Let $t \in T$ and $g \in G$.
The commutativity of \eqref{HNN diagram 1} 
gives
$$
\psihat(g^{-1}t\rho(g)) = 
\psi_G(g)^{-1} \psi_T(t) \lambda(\rho(g)) =
\psi_G(g)^{-1} \psi_T(t) \rho'(\psi_G(g))
$$
and, as $\psi$ is a morphism of piles,
$\psihat(t^{\trsl g}) = \psi_T(t)^{\trsl \psi_G(g)}$.
Therefore $\psihat$ preserves relations \eqref{relation of pile HNN}
and hence induces a homomorphism
$\psibar \colon \Gbar \to \Gbar'$.
Diagram \eqref{HNN diagram 2} commutes because $\psihat$ is
$\psi_G$ on $G$,
$\lambda$ on $L$,
and $\psi_T(T) \subseteq T'$.
\end{proof}

\begin{cor}\label{inverse limit}
Let $\{\bfG_i = (G_i, T_i)\}_{i\in I}$ be an inverse system of piles
and $\bfG = (G,T) = \varprojlim_i \bfG_i$.
Let $\{\rho_i \colon G_i \to L\}_{i \in I}$ 
be a compatible system of homomorphisms,
and let $\rho \colon G \to L = \varprojlim_i \rho_i$.
Assume that every $\rho_i$ is injective on every $(G_i)_{t_i}$
and hence
$\rho$ is injective on every $G_t$.
Then
$\pHNN(G,T,\rho,L) = \varprojlim_i \pHNN(G_i,T_i,\rho_i,L)$.
\end{cor}

\begin{rem}\label{HNN to pHNN}
In \eqref{relation of pile HNN},
if $g \in G_t$, then 
$t^{\trsl g} = t$,
so \eqref{relation of pile HNN} reads
$g^{-1}t\rho(g)=t$,
that is,
$\rho(g)= g^t$,
relation~\eqref{HNN relations}.
Thus for a closed subset $T_0$ of $T$ the identity maps of $G$, $L$ and $T_0 \to T$ induce
a homomorphism
$\theta \colon \ourHNN(G,T_0,\rho|_{T_0},L) \to \pHNN(G,T,\rho,L)$.
If $T =\{t^{\dot g}\mid t\in T_0,g\in G\}$, then  $\theta$ is an epimorphism.
By \eqref{relation of pile HNN},
$\Ker \theta$
is the normal closure of 
$\{(t^{\trsl g})^{-1} g^{-1} t \rho(g) \suchthat t \in T_0,\, g \in G\}$.
\end{rem}

\begin{lem}\label{with section}
Let $\bfG=(G,T)$ be a pile of pro-$\C$ groups.
Suppose that
there is a closed set
$T_0$ of representatives of the $G$-orbits of $T$.
Let 
$\Gbar=\pHNN(G,T,\rho,L)$
and
$\tG = \ourHNN(G,T_0,\rho|_{T_0},L)$.
Then the identities of $G$, $L$, and $T_0$ induce
an isomorphism $\tG \to \Gbar$.
\end{lem}

\begin{proof}
By Remark~\ref{HNN to pHNN},
these identities induce an epimorphism $\theta \colon \tG \to \Gbar$.

Conversely,
define a continuous map $T \to \tG$ by
$t^{\trsl g} \mapsto g^{-1}t\rho(g)$,
for $t \in T_0$ and $g \in G$.
It is well defined:
If $t^{\trsl g} = t$,
then $g \in G_t$,
hence
$\rho(g)= g^t$ in $\tG$,
that is,
$g^{-1}t\rho(g)=t$;
it follows that if 
$t^{\trsl g} = t^{\trsl h}$,
then
$g^{-1}t\rho(g) = h^{-1}t\rho(h)$ in $\tG$.
Thus this map, together with the identities of $G$ and $L$,
defines a homomorphism $\lambda \colon \Gbar \to \tG$.
Notice that $\lambda$ is the identity on $T_0$
and we have
$\Gbar = \langle G, L, T_0 \rangle$, by \eqref{relation of pile HNN}.
Therefore $\lambda$ is the inverse of $\theta$.
\end{proof}

We shall later need the following simple

\begin{lem}\label{free product}
Let $G$ be a pro-$p$ group having a subgroup $L$. Suppose $G=\varprojlim_{i\in I} G_i$ such that $G_i=F_i\amalg L_i$, where $F_i$ is a free pro-$p$ group and $L_i$ is the image of $L$ in $G_i$. Then $G=F\amalg L$ for some free pro-$p$ group $F$.
\end{lem}

\begin{proof}
Put $K=L^G$. Then $K\triangleleft G$.
Note that $G/K=\varprojlim_{i\in I} G_i/L_i^{G_i}\cong \varprojlim_{i\in I} F_i$ is a free pro-$p$ group. Therefore the quotient map $G\to G/L^G$ splits. We denote by $F$ the image of this splitting. Then $G$ is the semidirect product of $F$ and $K$.
The natural embeddings of $F$ and $L$ into $G$ induce the cohomology homomorphisms
\begin{align}
H^1(G,\F_p) &\to H^1(F,\F_p)\oplus H^1(L,\F_p), \label{H1} \\
H^2(G,\F_p) &\to H^2(L,\F_p)\oplus H^2(L,\F_p). \label{H2}
\end{align}
We first claim that the map in \eqref{H1} is an isomorphism.

To show that it is injective,
let $\rho \colon G \to \F_p$ be a homomorphism such that
$\rho|_{F} = 0$ and $\rho|_{L} = 0$.
Then, for all $x \in L$ and $\sigma \in G$, we have
$\rho(x^\sigma) = -\rho(\sigma)+ \rho(x) +\rho(\sigma) = \rho(x) = 0$,
hence $\rho|_{K} = 0$.
Since also $\rho|{F} = 0$ and $G = FK$, we have $\rho = 0$.

To show that the map in \eqref{H1} is surjective,
let $\varphi \colon F \to \F_p$ and $\psi \colon L \to \F_p$ be two homomorphisms.
We have to show that $\varphi, \psi$ extend to a homomorphism
$\rho \colon G \to \F_p$.

There is $i \in I$ such that $\psi$ factors as $\psi = \bar\psi \circ \pi_{i,L}$,
where
$\pi_{i,L} \colon L \to L_i$ is the restriction to $L$ of the map
$\pi_i \colon G \to G_i$ of the inverse system
and $\bar\psi \colon L_i \to \F_p$ is a homomorphism.
Since $G_i=F_i\amalg L_i$, we can extend $\bar\psi$ to a homomorphism $\bar\psi' \colon G \to \F_p$.
Then $\psi' = \bar\psi' \circ \pi_i \colon G \to \F_p$ is a homomorphism that extends $\psi$.

For all $x \in L$ and $\sigma \in G$ we have
$\psi(x^\sigma) = \psi(x)^{\psi(\sigma)} = \psi(x)$.
Therefore,
$\psi(x^\sigma) = \psi(x)$,
for all $x \in K$ and $\sigma \in G$.
Finally, we extend $\psi|_{K}$ and $\varphi$ to a homomorphism
$\rho \colon G=FK \to \F_p$ by
$\rho(yx) = \varphi(y) \psi(x)$, for $y \in F$ and $x \in K$.

Now, the map in \eqref{H2} is an isomorphism as well.
Indeed, since $G=\varprojlim_{i\in I} G_i$, $L=\varprojlim_{i\in I} L_i$,
$H^2(F,\F_p)= 0$, and $H^2(F_i,\F_p) = 0$ for every $i \in I$,
the map in \eqref{H2} is the direct limit of the homomorphisms
$$
H^2(G_i,\F_p)\to H^2(L_i,\F_p)\oplus H^2(F_i,\F_p),
$$
that are isomorphisms by \cite[Theorem 9.3.10]{RZ}. Hence by \cite[Theorem 9.3.10]{RZ} again, $G=F\amalg L$.
\end{proof}

\section{HNN-extensions of relatively projective pro-$p$ groups}

In this section we fix a prime $p$.
For a profinite space $T$
let $F(T)$ denote the free pro-$p$ group on $T$.

Let $\bfG=(G,T)$ be a projective pile of pro-$p$ groups
and let $\rho\colon G\to L$ be a homomorphism of pro-$p$ groups,
injective on $G_t$, for every $t \in T$.

Let
\begin{equation}\label{mainHNN}
\tG =
\ourHNN(G,T,\rho,L)
\quad \text{and} \quad
\Gbar =
\pHNN(G,T,\rho,L).
\end{equation}
Let 
$\trho \colon \tG\to L$
and
$\bar\rho \colon \Gbar\to L$
be the HNN-extensions of $\rho$ (see Definitions \ref{normal HNN} and \ref{pileHNN}). 
Let $F = \Ker \rho$,
$\tF = \Ker \trho$, and
$\Fbar = \Ker \bar\rho$.
Then 
\begin{equation}\label{semidirect with L}
\tG = \tF \rtimes L
\qquad \text{ and } \qquad
\Gbar = \Fbar \rtimes L.
\end{equation}
By \cite[Lemma 10]{HZ 07}, $\tF$ is a free pro-$p$ group.

\begin{lem}\label{mod L}
$\Gbar/L^{\Gbar} \cong G/\langle G_t\mid t\in T\rangle \amalg F(T/G)$.
\end{lem}

\begin{proof}
Let $K = \langle G_t \suchthat t\in T\rangle$.
This is a normal subgroup of $G$.
Using the presentation \eqref{relation of pile HNN}
we see that
$\Gbar/L^{\Gbar} = (G \amalg F(T))/N$,
where
$N = \langle t^{\trsl g}t^{-1}g \suchthat t\in T, \; g \in G \rangle^{G \amalg F(T)}$.

If $g \in G_t$, then 
$t = t^{\trsl g}$,
so $g \in N$,
and hence $K \le N$.
If
$g, h \in G$ satisfy $g \equiv h \pmod{K}$,
then $g \equiv h \pmod{N}$,
hence
$t^{\trsl g} \equiv g^{-1}t \equiv h^{-1}t = t^{\trsl h} \pmod{N}$.
Thus
$\Gbar/L^{\Gbar} = (G/K \amalg F(T/K))/(N/K)$.

As $K$ contains the stabilizers of all $t \in T$, the action
of $G/K$ on $T/K$ is with trivial stabilizers,
and hence $T/K$ has a closed set $\Tbar_0$
of representatives of the $G/K$-orbits (see \cite[Lemma 5.6.5]{RZ}).

Now note that the homomorphism
\begin{equation*}
G/K \amalg F(\Tbar_0) \to (G/K \amalg F(T/K))/(N/K)
\end{equation*}
has an inverse, induced from the homomorphism
$(G/K \amalg F(T/K)) \to G/K \amalg F(\Tbar_0)$ 
given by the identity of $G/K$
and mapping $t^{\trsl g} \mapsto g^{-1}t$,
for $t \in \Tbar_0$ and $g \in G/K$. So 
$\Gbar/L^{\Gbar} = G/K \amalg F(\Tbar_0)$.

The quotient map $T/K \to (T/K)/(G/K) = T/G$
maps $\Tbar_0$ homeomorphically onto $T/G$.
Hence we may write
$\Gbar/L^{\Gbar} \cong G/K \amalg F(T/G)$.
\end{proof}

\begin{lem}\label{section}
Assume that there exists
a closed set of representatives $T_0$ of the $G$-orbits in $T$.
Then
$\Gbar = E \amalg L \amalg F(T_0)$,
where
$E \le G$ is a free pro-$p$ group
isomorphic to $G/\langle G_t\mid t\in T\rangle$.
\end{lem}

\begin{proof}
By \cite[Theorem 9.5]{Haran},
$G = (\coprod_{t\in T_0} G_t) \amalg E$
for some free pro-$p$ group $E$.
By Lemma \ref{with section},
$\Gbar = \ourHNN(G, T_0, \phi|_{T_0},L)$.
From the presentation of this group
(see (\ref{ourHNN relations}))
it follows that
$\Gbar = E \amalg L \amalg F(T_0)$.

As 
$\langle G_t \suchthat t \in T \rangle$
is the kernel of the projection 
$(\coprod_{t\in T_0} G_t) \amalg E \to E$,
we have
$E \cong G/\langle G_t\mid t\in T\rangle $.
\end{proof}

\begin{lem}\label{HNN EP}
Assume that 
$L$ is finite and let $\mathcal{P}$ be a partition of $T$.
Let $A$ be a finite $p$-group
and let $\varphi \colon G \to A$ be an epimorphism.
Then
\begin{itemize}
\item[(a)]
there exist
a commutative diagram
of piles
\begin{equation}\label{triangle}
\xymatrix{
& \bfG \rlap{$\ =(G,T)$} \ar[d]_{\varphi} \ar[dl]_{\psi}
\\
\llap{$(B,Y)=\ $} \bfB \ar[r]^{\alpha} & \bfA \rlap{$\ =(A,X)$}
}
\end{equation}
in which
$\bfA = (A,X)$ is a finite pile, 
$\bfB = (B,Y)$ is a basic pro-$p$ pile (see Example \ref{basic pile}), 
and
$\alpha \colon \bfB \to \bfA$ is a rigid epimorphism.
Moreover,
$\varphi(T) = X$,
and the partition
$\{\varphi_T^{-1}(x) \suchthat x \in X\}$
is finer than $\mathcal{P}$.
\item[(b)]
Let $\rho_A \colon A \to L$ be an epimorphism
such that $\rho = \rho_A \circ \varphi$
and put
$\rho_B = \rho_A \circ \alpha$.
Let
$\Abar = \pHNN(A, X, \rho_A, L)$
and
$\Bbar = \pHNN(B, Y, \rho_B, L)$.
Then there is a commutative diagram
\begin{equation}\label{induced diagram}
\labelmargin-{2.5pt}
\xymatrix@=8pt{
&&& G \ar[dr]^{\zeta_G} \ar[ddd]_(.4){\varphi}|!{[dd]}\hole \ar[dddlll]_{\psi}
\\
&&&& \Gbar \ar[ddd]_(.4){\bar\varphi} \ar@/^1pc/[rdddd]^{\bar\rho}
\ar@<-.5pt>[dddlll]_{\bar\psi}
\\
&&&
\\
B \ar[dr]_{\zeta_B} \ar[rrr]^(.4){\alpha}|!{[rr]}\hole 
&&& A \ar[dr]_<{\zeta_A} 
& {}
\\
& \Bbar \ar[rrr]_(.4){\bar\alpha} 
\ar@/_1pc/[rrrrd]_{\bar\rho_B}
&&& \Abar \ar[rd]_{\bar\rho_A}
\\
&&&&& L
}
\end{equation}
in which $\bar\rho, \bar\rho_A, \bar\rho_B$
are the HNN-extensions of
$\rho, \rho_A, \rho_B$,
respectively.
\end{itemize}
\end{lem}

\begin{proof}
Let $\G = \{G_t \suchthat t \in T\}$.
Let $A_1,\ldots, A_n$ be representatives
of all (not necessarily distinct) conjugacy classes of
$\varphi(\G) \subseteq \Subgr(A)$. 
We can view them as indexed by the discrete set
$X_0 = \{1,\ldots, n\}$.
Let $X$ be the standard $A$-extension of $X_0$.
Then $\bfA = (A,X)$ is a finite pile,
with $\{A_x \suchthat x \in X\} = \varphi(\G)$.

Replacing $X$ with the union of suitably many copies of $X$,
if necessary,
we may assume that
$|\{x \in X \suchthat A_x = A'\}|$
is sufficiently large, for every $A' \in \varphi(\G)$.
Thus, by Corollary~\ref{completion of homomorphism},
we can complete $\varphi$
(by a continuous map $\varphi_T \colon T \to X$)
to a morphism of piles
$\varphi \colon \bfG \to \bfA$
such that the partition
$\{\varphi_T^{-1}(x) \suchthat x \in X\}$
is finer than $\mathcal{P}$.
Replace $X$ by $\varphi_T(T)$
to assume that $\varphi_T$ is surjective.

Let $P_B$ be a free pro-$p$ group
with an epimorphism $\alpha' \colon P_B \to A$
and let $B_x$ be a copy of $A_x$, for each $x \in X_0$.
Form the free pro-$p$ product $B = P_B \amalg (\coprod_{x \in X_0} A_x)$,
and let $\alpha \colon B \to A$
be the epimorphism that extends $\alpha'$
and is identity on each $A_x$.
Let $Y$ be the standard $B$-extension of $X_0$.
Then $\bfB = (B,Y)$ is a pile
and the identity of $X_0$ extends $\alpha$ to a rigid epimorphism
$\alpha \colon \bfB \to \bfA$.

By Proposition \ref{general EP}
there is a morphism $\psi \colon \bfG \to \bfB$
such that $\alpha \circ \psi = \varphi$.

(b)
By assumption
$\rho_A = \rho \circ \varphi$
and $\rho_B = \rho_A \circ \alpha$.
Hence also
$$
\rho =
\bar\rho \circ \zeta_G =
\pi \circ \bar\lambda \circ \zeta_G =
\rho_A \circ \varphi =
\rho_A \circ \alpha \circ \psi =
\rho_B \circ \psi.
$$
The existence of $\bar\varphi$, $\bar\alpha$, and $\bar\psi$
follows from this by Lemma~\ref{functorial},
with the pile morphism
$\varphi$ (resp. $\alpha$, $\psi$)
and the identity of $L$.
The commutativity of \eqref{triangle}
ensures that \eqref{induced diagram} is commutative.
\end{proof}

\begin{lem}\label{finite pile HNN proper}
Assume that $L$ is finite.
Then the HNN-map $\zeta_G \colon G \to \Gbar$ 
is injective.
\end{lem}

\begin{proof}
First assume that $\bfG$ is basic
(Example \ref{basic pile}).
By Lemma \ref{with section},
$\Gbar$ is an HNN'-extension of $G$,
that is, 
an HNN-extension of $G \amalg L$.
Notice that $G \amalg L = \Ker(\rho_L) \rtimes L$,
where $\Ker(\rho_L)$ is a free pro-$p$ group,
by Lemma~\ref{Ker projective},
and so $\zeta_G$ is injective, by \cite[Lemma A.1]{permhnn}.

In the general case
it suffices to show that $\Ker(\zeta_G)$
is contained in every open normal subgroup $K$ of $G$.

Put $A = G/K$ and let $\varphi \colon G\to A$ be the quotient map.
Replacing $K$ with $K \cap F$ we may assume that $K \le F$,
and hence there is an epimorphism $\pi \colon A \to L$
such that $\rho = \pi \circ \varphi$.
By Lemma \ref{HNN EP} there are commutative diagrams
\eqref{triangle} and \eqref{induced diagram}.

By the first paragraph of this proof
$\zeta_B$ is injective, hence by \eqref{induced diagram}
$$
\Ker \zeta_G \le \Ker \psibar \circ \zeta_G = 
\Ker \zeta_B \circ \psi = \Ker \psi.
$$
By \eqref{triangle},
$\Ker \psi \le \Ker \varphi = K$.
Thus $\Ker \zeta_G \le K$.
\end{proof}

\begin{lem}\label{finite pileHNN structure}
Assume that $L$ is finite.
Then 
\begin{equation}\label{mod R}
\Gbar = L \amalg E,
\end{equation}
where $E \cong G/\langle G_t \suchthat t\in T\rangle \amalg F(T/G)$.
\end{lem}

\begin{proof}
If $\bfG$ is basic, this is Lemma~\ref{section}.

\begin{claim} 
$\tG$ is a special HNN-extension
with base group $L$, say,
\begin{equation}\label{mainHNNspecial}
\tG=\HNN(L, S, \{L_s\}_{s \in S}).
\end{equation}
\end{claim}

\noindent
Proof of Claim 1.
This follows from \eqref{semidirect with L} by Theorem \ref{permext},
provided
every finite subgroup $L_0$ of $\tG$ is conjugate to a subgroup of $L$.

Let $G_L = G\amalg L$.
By \cite[Lemma A.1]{permhnn},
the map $G_L \to \tG$ is injective,
hence, by \cite[Theorem 7.1.2]{R 2017},
$L_0$ is contained in a conjugate of $G_L$.
Hence, by \cite[Theorem A]{HR 85},
$L_0$ is contained in a conjugate of either $L$ or $G$.
In the latter case 
$L_0$ is contained in a conjugate of $G_t$, for some $t \in T$,
by \cite[Proposition 5.4.3]{HJ},
and thus, as $G_t^t \le L$, again contained in a conjugate of $L$.

\medskip
By Remark~\ref{choice stable letters}
we may assume that
$\tF = \langle S \rangle^{\tG}$.

\medskip
Throughout the rest of the proof we use the following notation:
\begin{align*}
T' &= \{t\in T \suchthat G_t \neq 1\},
\\
S' &= \{s\in S \suchthat L_s \neq 1\},
\quad
S''= S \smallsetminus S'
\\
R_0 &= 
\langle (t^{\trsl g})^{-1} g^{-1} t \rho(g)
\suchthat \allowbreak
{t \in T',\, g \in G \rangle^{\tG}},
\\
R &= \langle C_{\tF}(G_t) \suchthat t\in T' \rangle^{\tG}.
\\
R' &= \langle C_{\tF}(L_s)) \suchthat s \in S' \rangle^{\tG}.
\end{align*}

\begin{claim} 
$R_0 = R = R' = \langle S' \rangle^{\tG}$.
\end{claim}

Proof of Claim 2.

Let $t \in T'$ and $g \in G$.
Let $g_t \in G_t$.
Then $g_t^g \in G_t^g = G_{t^{\trsl g}}$.
By \eqref{HNN relations},
$\rho(g_t) = g_t^t$
and
$\rho(g_t^g) = (g_t^g)^{t^{\trsl g}}$,
hence
$$
g_t^{g \, t^{\trsl g}} = \rho(g_t^g) =
\rho(g_t)^{\rho(g)} = g_t^{t\, \rho(g)} ,
$$
whence
$(t^{\trsl g})^{-1} g^{-1} t \rho(g) \in \tF$ centralizes $G_t$,
and therefore is in $R$.
As $R \normal \tG$, this proves $R_0 \le R$.

By Lemma \ref{special properties}(c),
there is an epimorphism $\tG \to L \amalg F(S'')$
with kernel $\langle S' \rangle^{\tG}$, which is $R'$,
by Lemma \ref{special properties}(b).
This epimorphism is injective on $L$, and hence on every $G_t$,
since $G_t$ is conjugate to a subgroup of $L$ in $\tG$.
Thus the image of $G_t$ is a non-trivial finite subgroup of $L$.
The centralizer in $L \amalg F(S'')$
of any non-trivial element of $L$
is contained in $L$, by \cite[Theorem B]{HR 85},
and hence trivially intersects $\langle S''\rangle^{L \amalg F(S'')}$,
the image of $\tF = \langle S \rangle^{\tG}$
in $L \amalg F(S'')$.
Hence $C_{\tF}(G_t) \le R'$ for every $t \in T'$,
whence $R \le R'$.

To show that $R' \le R_0$
--as $R_0$ is,
by Remark~\ref{HNN to pHNN},
the kernel of 
the canonical epimorphism $\theta_G \colon \tG \to \Gbar$--
we have to prove that
$\theta_G(R') = 1$,
that is, that
$\theta_G(C_{\tF}(L_s)) = 1$, for every $s \in S'$.
Since $\theta_G(\tF) = \Fbar$
and since $\theta_G$ is the identity on $L$,
it suffices to show that
$C_{\Fbar}(L_s) = 1$, for every $s \in S'$.

Fix $s \in S'$.
We only use that $L_s$ is a finite nontrivial subgroup of $L$.

If $\bfG$ is basic,
then, by Lemma \ref{section},
$\Gbar$ is the free pro-$p$ product of $L$ with a free pro-$p$ group.
As $L_s$ is finite,
by \cite[Theorem A]{HR 85}
$L_s \le L^{\bar g}$ for some $\bar g \in \Gbar$,
and, as $L_s \ne 1$,
by \cite[Theorem B]{HR 85}
$C_{\Gbar}(L_s) \le L^{\bar g}$.
But $L \cap \Fbar = 1$,
hence $L^{\bar g} \cap \Fbar = 1$,
whence $C_{\Fbar}(L_s) = 1$,
as contended.

In the general case
we show that
$C_{\Fbar}(L_s)$ is contained in
every open normal subgroup $\Mbar$ of $\Gbar$.
Let $\bar\lambda \colon \Gbar \to C := \Gbar/\Mbar$
be the quotient map;
we show that
$\bar\lambda(C_{\Fbar}(L_s)) = 1$.

As $\Fbar$ is open in $\Gbar$, without loss of generality $\Mbar \le \Fbar$,
that is,
$\Ker \bar\lambda \le \Ker \bar\rho$.
Thus there is an epimorphism
$\pi \colon C \to L$
such that
$\pi \circ \bar\lambda = \bar\rho$.
Let $\varphi \colon G \to A$ be the epimorphism
$\varphi = \bar\lambda \circ \zeta_G$,
where $A :=\varphi(G) \le C$.


Let $\rho_A = \pi|_A \colon A \to L$
and $\rho_B = \rho_A \circ \alpha \colon B \to L$.
Then
$$
\rho =
\bar\rho \circ \zeta_G =
\pi \circ \bar\lambda \circ \zeta_G =
\rho_A \circ \varphi =
\rho_A \circ \alpha \circ \psi =
\rho_B \circ \psi.
$$

By Lemma \ref{HNN EP}
we can complete $A$ to a finite pile
$\bfA = (A,X)$ 
and complete $\varphi$
(by a continuous surjective map $\varphi_T \colon T \to X$)
to a morphism of piles
$\varphi \colon \bfG \to \bfA$
and obtain commutative diagrams
\eqref{triangle}
and
\eqref{induced diagram}
with a basic pile $\bfB$.
Moreover, the partition
$\{\varphi_T^{-1}(x) \suchthat x \in X\}$
is finer than the partition induced by
$\bar\lambda \circ \zeta_T \colon T \to C$,
whence there is a map
$\mu \colon X \to C$
such that
$\mu \circ \varphi_T = \bar\lambda \circ \zeta_T$.

Since $\Abar$ satisfies the same defining relations as $\Gbar$,
the homomorphism
$\hat\mu \colon A \amalg L \amalg F(X) \to C$,
which is the inclusion $A \to C$ on $A$, $\bar\lambda|_L$ on $L$, and $\mu$ on $X$,
induces a homomorphism $\bar\mu \colon \Abar \to C$,
such that
$\bar\mu \circ \bar\varphi = \bar\lambda$.
%

Thus we have the following commutative diagram 
$$
\xymatrix{
 \Gbar \ar[d]_{\bar\psi} \ar@{=}[r] 
& \Gbar \ar[d]_{\bar\varphi} \ar@{=}[r] 
& \Gbar \ar[d]_{\bar\lambda} \ar[dr]^{\bar\rho}
\\
 \Bbar \ar[r]^{\bar\alpha} 
& \Abar \ar[r]^{\bar\mu}
& C \ar[r]^{\pi} & L
\\}
$$
In particular,
$\bar\lambda$ factors through $\bar\psi$.
So, to show 
$\bar\lambda(C_{\Fbar}(L_s)) = 1$,
it suffices to show that
$\bar\psi(C_{\Fbar}(L_s)) = 1$.
As $\bar\lambda$ is injective on $L$,
so is $\bar\psi$, hence
$D := \bar\psi(L_s)$ is a finite nontrivial subgroup of $\Bbar$.

Let $\bar\rho_B \colon \Bbar \to L$
be the HNN-extension of
$\rho_B = \rho_A \circ \alpha \colon B \to L$
and let $\Fbar_B = \Ker(\bar\rho_B)$.
Then
$\bar\rho = \bar\rho_B \circ \bar\psi$,
whence $\bar\psi(\Fbar) \le \Fbar_B$.
Thus $\bar\psi(C_{\Fbar}(L_s)) \le C_{\Fbar_B}(D)$.
But $C_{\Fbar_B}(D) = 1$,
as shown above in the case of a basic pile.
This finishes the proof of 
$R' \le R_0$.

Finally, by Lemma \ref{special properties}(b),
$R' = \langle S' \rangle^{\tG}$.

\begin{claim} 
$\Gbar=L\coprod F(S'')$.
\end{claim}

Indeed,
by Remark~\ref{HNN to pHNN},
$\Gbar = \tG/R_0$,
so, by Claim 2, 
$\Gbar = 
\tG/ \langle S' \rangle^{\tG}$.
By Lemma \ref{special properties}(c),
$\Gbar = L \amalg F(S'')$.

\bigskip

Finally, by Claim 3,
$\Gbar/L^{\Gbar} \cong F(S'')$,
so, by Lemma \ref{mod L},
$F(S'') \cong G/K \amalg F(T/G)$.
\end{proof}

\begin{thm}\label{proper}\hspace{9cm}
\begin{enumerate}
\item[(i)]
$\tG$ is proper, i.e.,
the map
$\xi \colon G \amalg L \to \tG$ is injective.
Hence also $\eta_G \colon G \to \tG$ is injective.

\item[(ii)]
$\zeta_G \colon G \to \Gbar$ is injective.

\item[(iii)]
$\Gbar = L \amalg E$,
where $E \cong G/\langle G_t \suchthat t\in T\rangle \amalg F(T/G)$.
\end{enumerate}
\end{thm}

\begin{proof}
Let $\U$ be the set of open normal subgroups of $L$.
For every $U \in \U$ let $N_U=\rho^{-1}(U)$.
Since $T$ is closed under $G$-conjugation,
the subgroup $\tN_U=\langle N_U\cap G_t \suchthat t\in T\rangle$
is normal in $G$.

\begin{step}
$\bigcap_{U\in\U} \tN_U = \{1\}$.
\end{step}

Indeed,
let $V$ be an open subgroup of $G$.
As $\{G_t\}_{t\in T}$ is a continuous family,
$\bigcup_{t \in T} G_t$ is a closed subset of $G$ (see Lemma \ref{continuous family}).
So is the kernel $F$ of $\rho$.
Since $\rho_{|G_t}$ is injective for every $t\in T$,
we have
$F \cap (\bigcup_{t \in T} G_t) = \{1\}$.
Also
$F = \bigcap_{U\in\U} N_U$,
so
$\bigcap_{U\in\U} \big(N_U \cap (\bigcup_{t \in T} G_t) \big) = \{1\} \subseteq V$.
By compactness there is $U$ such that
$N_U \cap (\bigcup_{t \in T} G_t) \subseteq V$.
In particular,
$N_U \cap G_t \subseteq V$ for every $t \in T$.
Therefore
$\tN_U = \langle G_t \cap N_U \suchthat t\in T\rangle \subseteq V$.
Since $V$ is an arbitrary open subgroup of $G$, this implies that
$\bigcap_{U\in\U} \tN_U = \{1\}$.

Set
$L_U = L/U$,
$G_U = G/\tN_U$,
$T_U = T/\tN_U$,
and $\bfG_U = (\G_U, T_U)$.
Then
$(G_U)_{t_U} := \Stab_{G_U}(t_U) = 
G_t\tN_U/\tN_U$,
where $t \in T$ is a representative of $t_U \in T_U$.

It follows from the Claim that
$\rho=\varprojlim_{U\in\U}\rho_U$,
where $\rho_U \colon G_U\to L_U$,
for every $U \in \U$,
is the homomorphism induced from $\rho$.
Clearly, $\Ker \rho_U = N_U/\tN_U$.
Let $t \in T$ and let $t_U$ be its image in $T_U$.
By the definition of $\tN_U$ we have $G_t \cap \tN_U = G_t \cap N_U$,
hence
$N_U \cap G_t \tN_U = \tN_U$,
whence $\rho_U$ is injective on $(G_U)_{t_U}$.
Thus we may define
\begin{align}
\tG_U &= \ourHNN(G_U, T_U, \rho_U, L_U),\label{quotientHNN}
\\
\Gbar_U &= \pHNN(G_U, T_U, \rho_U, L_U).\label{quotientpileHNN}
\end{align}

By Lemma \ref{mod tilde}, $\bfG_U$ is projective
and by Corollary \ref{projectivity mod tilde}, $N_U/\tN_U$ is free pro-$p$.
Therefore,
by \cite[Lemma A.1]{permhnn},
\eqref{quotientHNN}
is a proper HNN-extension,
i.e., the map $\xi_U \colon G_U \amalg L \to \tG_U$ is injective.

We have commutative diagrams
\begin{equation}
\xymatrix{
G \amalg L \ar[r]^{\xi}\ar[d]& \tG\ar[d]\\
G_U \amalg L \ar[r]^{\xi_U} &\tG_U\rlap{$,$}
}
\qquad
\xymatrix{
G \ar[r]^{\zeta_G}\ar[d]& \Gbar\ar[d]\\
G_U \ar[r]^{\zeta_{G_U}} &\Gbar_U\rlap{$,$}
}
\end{equation}
where the right vertical maps are
induced from the quotient maps $G \to G_U$ and $T \to T_U$.

\medskip
(i) 
Since $\xi=\varprojlim_{U\in\U}\xi_U$,
we deduce from Corollary \ref{inverse limit} that $\xi$ is injective.

\medskip
(ii) 
By Lemma \ref{finite pile HNN proper}
the HNN-map $\zeta_{G_U}\colon G_U\to \Gbar_U$ is injective.
Since in the above diagram
$\zeta_G=\varprojlim_{U\in\U}\zeta_{G_U}$,
we deduce that $\zeta_G$ is injective.
 


(iii)
Put $K = \langle G_t \suchthat t \in T\rangle$
and $K_U = \langle (G_U)_{t_U} \suchthat t_U\in T_U\rangle$,
for every $U$.
By Lemma \ref{finite pileHNN structure},
$\Gbar_U = L_U \amalg E_U$,
where
$E_U \cong G_U/K_U \amalg F(T_U/G_U)$.
By Lemma \ref{kill all stabilizers},
$G_U/K_U$, and hence also
$E_U$,
is free pro-$p$.
But
$G = \varprojlim_U G_U$,
$L = \varprojlim_U L_U$,
and
$\Gbar = \varprojlim_U \Gbar_U$,
hence by Lemma \ref{free product}
there is a free pro-$p$ group $E$
such that
$\Gbar = L \amalg E$.

It follows that
$E \cong \Gbar/L^{\Gbar}$,
hence by Lemma \ref{mod L},
$E \cong G/K \amalg F(T/G)$.
\end{proof}

In the next lemma we identify $G$ with its image
in $\Gbar$ via the embedding $\zeta_G$.

\begin{lem}
Let $\sigma \in \Gbar$ such that
$L^\sigma \cap G \ne 1$.
Then there is a unique $t \in T$ such that
$L^\sigma \cap G = G_t$ in $\Gbar$.
It satisfies $\sigma t \in L$.
Moreover,
\begin{equation}\label{restriction to subgroup}
\{L^\sigma \cap G \suchthat \sigma \in \Gbar \}
\smallsetminus \{1\}
=\{G_t \suchthat t \in T\}
\smallsetminus \{1\}.
\end{equation}
\end{lem}

\begin{proof}
\begin{claim} 
$L^{t^{-1}} \cap G \ge G_t$ in $\Gbar$.
\end{claim}

Indeed,
$G_t^t = \rho(G_t) \in L \le \Gbar$, so
$$
L^{t^{-1}} \cap G \ge
L^{t^{-1}} \cap G_t =
(L \cap \rho(G_t))^{t^{-1}} = 
\rho(G_t)^{t^{-1}} = 
G_t.
$$

\begin{claim} 
If $L^\sigma \cap G \le G_t$,
then $L^\sigma \cap G = G_t$
and $\sigma t \in L$.
\end{claim}

Indeed,
by the assumption,
$L^\sigma \cap G = L^\sigma \cap G_t$,
so
$$
1 \ne (L^\sigma \cap G)^t =
(L^\sigma \cap G_t)^t =
L^{\sigma t} \cap \rho(G_t) =
(L^{\sigma t} \cap L) \cap \rho(G_t).
$$
In particular,
$L^{\sigma t} \cap L \ne 1$ in $\Gbar = L \amalg E$
(Theorem \ref{proper}(iii)).
By \cite[Lemma 3.1.10]{HJ}
this implies that
$\sigma t \in L$.
Therefore
$L^\sigma = L^{t^{-1}}$, hence
$L^\sigma \cap G = G_t$,
by Claim 1.

\medskip
Now, if $L$ is finite,
then $L^\sigma \cap G$
is a finite subgroup of $G$.
By \cite[Proposition 5.4.3]{HJ}
there is $t \in T$ such that
$L^\sigma \cap G \le G_t$.
So $L^\sigma \cap G = G_t$ by Claim 2.
By Proposition \ref{pile vs group projectivity}(b)
such $t$ is unique.

In the general case we have
$G = \varprojlim_U G_U$,
$T = \varprojlim_U T_U$,
$L = \varprojlim_U L_U$,
and
$\Gbar = \varprojlim_U \Gbar_U$,
with finite $L_U$,
as in Theorem \ref{proper}.
Let $\pi_U \colon \Gbar \to \Gbar_U$
and $\pi_{U,V} \colon \Gbar_V \to \Gbar_U$,
for $V \le U$,
be the maps of the inverse system.
We may assume that
$\pi_U(L^\sigma \cap G) \ne 1$ for every $U$.
Since $\pi_U(L^\sigma \cap G) \le L_U^{\pi_U(\sigma)} \cap G_U$,
we have
$L_U^{\pi_U(\sigma)} \cap G_U \ne 1$.
So by the preceding case there is a unique $t_U \in T_U$
such that
$L_U^{\pi_U(\sigma)} \cap G_U = (G_U)_{t_U}$
for every $U$.

If $V \le U$ and $\pi = \pi_{U,V}$, then 
$L_V^{\pi_V(\sigma)} \cap G_V = (G_V)_{t_V}$
and
$\pi(L_V^{\pi_V(\sigma)}) = L_U^{\pi_U(\sigma)}$,
$\pi(G_V) = G_U$,
so
$(G_U)_{\pi(t_V)} = \pi((G_V)_{t_V}) \le (G_U)_{t_U}$.
By Proposition \ref{pile vs group projectivity}(b),
$\pi(t_V) = t_U$.

It follows that there is a unique $t \in T$ such that
$\pi_U(t) = t_U$ for every $U$.
Therefore
$\pi_U(L^\sigma \cap G) \le \pi_U(G_t)$.
Hence
$L^\sigma \cap G \le G_t$.
By Claim 2,
$L^\sigma \cap G = G_t$.
As $\pi_U(\sigma t) \in L_U = \pi_U(L)$ for every $U$,
we have
$\sigma t \in L$.

The above proves that the left-handed side of 
\eqref{restriction to subgroup}
is contained in the right-handed side.
Conversely, let $t \in T$ such that $G_t \ne 1$.
By Claim 1,
$G_t \le L^{t^{-1}} \cap G$.
In particular
$L^{t^{-1}} \cap G \ne 1$.
So there is a unique $s \in T$ such that
$L^{t^{-1}} \cap G = G_s$.
Thus,
$G_t \le G_s$.
By Proposition \ref{pile vs group projectivity}(b),
$s = t$.
Hence
$L^{t^{-1}} \cap G = G_t$.
\end{proof}

\begin{proof}[Proof of Theorem \ref{main small}]
By Proposition~\ref{pile vs group projectivity}
and Remark \ref{pile vs group projectivity b},
$\bfG = (G,T)$ is a projective pile.
By Theorem \ref{proper},
$\Gbar = L \amalg F$, where $F$ is free,
and $\zeta = \zeta_G \colon G \to \Gbar$ is injective.
The rest is \eqref{restriction to subgroup}.
\end{proof}

Theorem \ref{smain} simply follows from Theorem \ref{main small}
if one puts $\rho$ to be the identity map of $G$ to its copy $L$.

Theorem \ref{into free product} follows from Theorem \ref{proper}
observing that $\zeta_G(G_t)$ in its proof is conjugate
to $\rho(G_t)$ for every $t\in T$.

\bigskip
\begin{proof}[Proof of Theorem \ref{inverse limit of free products}]

(i) $\Leftrightarrow$ (ii):
By Proposition \ref{pile vs group projectivity}
and Remark \ref{pile vs group projectivity b}.

(ii) $\Rightarrow$ (iii):
Let $\zeta \colon G\to \Gbar := L\amalg F$
be the embedding of Theorem \ref{smain}.
Put $\That = \{L^\sigma \suchthat \sigma \in \Gbar\}$
and let $\Gbar$ act on $\That$ by conjugation.
Then $\Stab_{\Gbar}(L^\sigma) = L^\sigma$
and if $L^\sigma \ne L^\tau$, then
$L^\sigma \cap L^\tau = 1$
(\cite[Lemma 3.1.10]{HJ}).
Hence the action of $G$ on $\That$ via $\zeta$ is such that
$\Stab_G(L^\sigma) = L^\sigma \cap \zeta(G)$
and, by Theorem \ref{smain}, it satisfies (a) for $\That$.

Write $L=\varprojlim L_i$,
$F =\varprojlim F_i$
as inverse limits
of finite quotient groups of $L$
and finitely generated free quotients of $F$, respectively.
Put $\Gbar_i = L_i \amalg F_i$
and let $G_i$ be the image of $\zeta(G)$ in $\Gbar_i$.
Then $\Gbar_i$, and hence also $G_i$,
is countably generated.
Put $T_i = \{L_i^\sigma \suchthat \sigma \in \Gbar_i\}$,
for every $i$,
then $(G,\That) = \varprojlim_i (G_i,T_i)$.

By the pro-$p$ version of Kurosh Subgroup Theorem
\cite[Theorem 9.7]{Haran}
$G_i$ is a free pro-$p$ product
$\coprod_{x_i \in X_i} (G_{x_i}) \amalg F'_i$,
for some closed set $X_i$
of representatives of the $G_i$-orbits in $T_i$
and some free pro-$p$ group $F'_i$.

(iii) $\Rightarrow$ (i):
Put $\bfG = (G,\That)$
and $\bfG_i = (G_i,T_i)$, for every $i$.
We have to solve a finite embedding problem \eqref{EP} for $\bfG$.
But $\varphi$ factors via some $\bfG_i$.
This means that there exist
$\nu_i \colon \bfG \to \bfG_i$,
$\varphi_i \colon \bfG_i \to \bfA$,
such that $\varphi = \varphi_i \circ \nu_i$.
By \cite[Proposition 5.4.2]{HJ},
$G_i$ is $\{(G_i)_t \suchthat t \in T_i\}$-projective,
so, by Proposition ~\ref{pile vs group projectivity},
$\bfG_i$ is projective.
Therefore there exists a solution $\psi_i$ of
the embedding problem $(\alpha,\varphi_i)$.
Then $\psi_i \circ \nu_i$ is a solution of \eqref{EP}.

(i) $\Rightarrow$ (iv):
Let again $\zeta \colon G\to \Gbar := L\amalg F$
be the embedding of Theorem \ref{smain}.
Let $X$ be a basis of $F$, that is, 
a profinite subspace of $F$ such that $F = F(X)$.
Then $\Gbar$ can be viewed as the special HNN-extension
$\Gbar = \HNN(L,X,\caL)$,
where $\caL$
is a family of copies $L_x$ of $1$, for every $x \in X$.
Thus $\Gbar$ is the fundamental group
of a bouquet of groups $\Gamma$,
with $L$ attached to the unique vertex
and $X$ the space of edges, with $1$ attached to each of them
(\cite[Example 6.2.3(e)]{R 2017}).

Let $S$ be the standard pro-$p$ graph of $\Gamma$
(\cite[Section 6.3]{R 2017}).
By \cite[Corollary 6.3.6]{R 2017},
$S$ is a pro-$p$ tree.
By \cite[Lemma 6.3.2(b)]{R 2017}
$\Gbar$ acts on $S$,
with trivial edge stabilizers
and vertex stabilizers being conjugates of $L$.
Then the restriction of this action to $\zeta(G)$
gives the required action. 

(iv) $\Rightarrow$ (i):
Let $G$ be a pro-$p$ group that acts on a pro-$p$ tree $\Gamma$.
Let $V$ be its set of vertices
and let $U$ be an open normal subgroup of $G$.
Then $\tU=\langle U\cap G_v\mid v\in V\rangle$
is a normal subgroup of $G$
and $G_U=G/\tU$ acts on $\Gamma_U=\Gamma/\tU$.
Also $U$ acts on $\Gamma$, hence
by \cite[Proposition 4.1.1]{R 2017}
$\Gamma_U$ is a pro-$p$ tree.
Moreover, $U/\tU$ acts freely on $\Gamma_U$
and hence is free pro-$p$ \cite[Theorem 4.1.2]{R 2017}.
So $G_U$ is virtually free pro-$p$
and every nontrivial torsion element $g \in G_U$ belongs to
some vertex stabilizer $(G_U)_v$
by \cite[Theorem 4.1.8]{R 2017}.

In fact,
even $C_{G_U}(g)\leq (G_U)_v$,
since for any $c\in C_{G_U}(g)$,
$g(cv) = cv$, so
$1 \ne g \in (G_U)_v \cap (G_U)_{cv}$
 and since edge stabilizers are trivial,
by \cite[Corollary 4.1.6]{R 2017} we have
$v = cv$, that is, $c \in (G_U)_v$.

We can now apply \cite[Corollary 6.3]{permhnn}
to deduce the existence of an embedding
$\zeta_U \colon G_U \to H=G/U\amalg F$
for some free pro-$p$ group $F$.
Then by \ref{sub free prod}(d)
$G_U$ is
$\G_U = \{(G_U)_v \suchthat v\in V\}
= 
\{G\cap (G/U)^h \suchthat h \in H\}$-projective.
Since $(G,\G) = \varprojlim_U (G_U,\G_U)$,
$G$ is $\G$-projective,
because every finite embedding problem factors via some $(G_U,\G_U)$.

\end{proof}

\end{document}